\pdfoutput=1
\pdfoutput=1
\pdfoutput=1
\pdfoutput=1
\pdfoutput=1
\documentclass[a4paper,10pt]{article}
\usepackage{calc}
\usepackage[all]{xy}
\usepackage[centertags]{amsmath}
\usepackage{latexsym}
\usepackage{amsfonts}
\usepackage{graphicx}
\usepackage{tikz}
\usepackage{cases}
\usepackage{amssymb}
\usepackage{amsthm}
\usepackage{color}
\usepackage{fancyhdr}
\usepackage[dvips]{epsfig}
\usepackage{newlfont}
\usepackage[latin1,utf8]{inputenc}
\usepackage[english,french]{babel}
\usepackage{graphicx}
\usepackage{t1enc}
\usepackage[english,french]{babel}
\usepackage{fancybox}
\usepackage{graphicx}
\usepackage{t1enc}
\usepackage[french2]{minitoc}
\usepackage{mathrsfs}
\usepackage{amsfonts}
\usepackage{wasysym}
\usepackage{hyperref}
\allowdisplaybreaks
\usepackage{float}
\usepackage{authblk}
\numberwithin{equation}{section}
\theoremstyle{plain}
\newtheorem{theorem}{Theorem}[section]
\newtheorem{lemma}[theorem]{Lemma}
\newtheorem{remark}[theorem]{Remark}
\newtheorem{proposition}[theorem]{Proposition}
\newtheorem{fact}[theorem]{Fact}
\newtheorem{definition}[theorem]{Definition}

\newtheorem{example}{Example}
\newcommand{\dR}{dR}

\newcommand{\CD}{C_D}
\newcommand{\PCD}{P_{\CD}}

 \makeatletter
\newcommand{\thechapterwords}
{ \ifcase \thechapter\or 1\or 2\or 3\or 4\or 5\or
	6\or 7\or 8\or 9\or 10\or 11\fi}
\def\thickhrulefill{\leavevmode \leaders \hrule height 2ex \hfill \kern \z@}
\def\@makechapterhead#1{%
	\vspace*{15\p@}%
	{\parindent \z@ \centering \reset@font
		\thickhrulefill\quad
		\scshape  {\chapnumfont \@chapapp{}}{\chapnumfont \thechapterwords}
		\quad \thickhrulefill
		\par\nobreak
		\vspace*{15\p@}%
		\interlinepenalty\@M
		\hrule
		\vspace*{15\p@}%
		\huge {\bfseries  #1}\par\nobreak
		\par
		\vspace*{15\p@}%
		\hrule
		\vskip 15\p@
}}
\def\@makeschapterhead#1{%
	\vspace*{15\p@}%
	{\parindent \z@ \centering \reset@font
		\thickhrulefill
		\par\nobreak
		\vspace*{15\p@}%
		\interlinepenalty\@M
		\hrule
		\vspace*{15\p@}%
		\Huge \bfseries #1\par\nobreak
		\par
		\vspace*{15\p@}%
		\hrule
		\vskip 30\p@
}}

\DeclareFixedFont{\chapnumfont}{T1}{phv}{b}{n}{20pt}
\DeclareFixedFont{\chapchapfont}{T1}{phv}{b}{n}{16pt}
\DeclareFixedFont{\chaptitfont}{T1}{phv}{b}{n}{24.88pt}
\def\@makechapterhead#1{%
	\vspace*{15\p@}%
	{\parindent \z@ \centering \reset@font
		\thickhrulefill\quad
		\scshape {\chaptitfont\color[rgb]{0.00,0.50,1.00}\@chapapp{}}
		{\chapnumfont \thechapterwords}
		\quad \thickhrulefill
		\par\nobreak
		\vspace*{15\p@}%
		\interlinepenalty\@M
		\hrule
		\vspace*{15\p@}%
		{\Large\bfseries #1}\par\nobreak
		\par
		\vspace*{15\p@}%
		\hrule
		\vskip 30\p@
}}%

\title{On  the geometric quantization of  $\theta$-almost twisted Poisson manifold}
	\author[1]{Nasser Saipele Nansidi}
\author[2,3]{Bertuel Tangue Ndawa}
\author[1]{Joseph Dongho}
\affil[1]{Faculty of Science, University of Maroua, Cameroon.}
\affil[2]{University Institute of Technology, University of Ngaoundere, Cameroon.}
\affil[3]{Institut des Hautes \'{e}tudes Scientifiques, Universit\'{e} Paris-Saclay, France.}
\date{}
\begin{document}
\maketitle
\selectlanguage{english}

The authors declare that they have no known competing financial interests or personal relationships that could have appeared to influence the work reported in this paper.

This research received no specific grant from any funding agency in the public, commercial, or not-for-profit sectors.

	\paragraph{Abstract}

We introduce and investigate the concept of a $\theta$-almost twisted Poisson manifold $(M,\Lambda,\varphi,\theta)$.  This structure consists of a smooth manifold $M$ equipped with a bivector field $\Lambda$, a 3-form $\varphi$, and a closed 1-form $\theta$, satisfying the following conditions: the exterior derivative $d\varphi$ of $\varphi$ equals the wedge product $\theta \wedge \varphi$; the anchor $\Lambda^\#(\theta)$ of $\theta$ vanishes identically; and one-half of the Schouten-Nijenhuis bracket $[\Lambda, \Lambda]$ equals the anchor $\Lambda^\#(\varphi)$ of $\varphi$. This structure generalizes both Poisson and twisted Poisson manifolds, permitting the 3-form $\varphi$ to be non-closed in a way controlled by the 1-form $\theta$.
We construct a Lie-Rinehart algebra on the module of 1-forms $\Omega^1(M)$, giving rise to a cochain complex and an associated cohomology theory called $\theta$-almost twisted Poisson cohomology.  Moreover, we develop the geometric quantization of these manifolds by defining a suitable contravariant derivative, establishing a prequantization condition in terms of the cohomology, and constructing a quantum Hilbert space via polarization. We illustrate our results with several examples, including the computation of the cohomology and quantization on $\mathbb{R}^5$.
	\medskip
	

	\section{Introduction}

Poisson structures, introduced by Sim\'{e}on Denis Poisson in the early 19th century, are fundamental objects in symplectic geometry and Hamiltonian mechanics. They are defined by a bivector field $\Lambda$ satisfying the Jacobi identity $[\Lambda,\Lambda]=0$, where $[\ ,\ ]$ denotes the Schouten-Nijenhuis bracket, see \cite{R4}.

In 2001, twisted Poisson structures \cite{R11} emerged as a significant generalization, relaxing the Jacobi identity to $[\Lambda,\Lambda] =2 \Lambda^{\#}(\varphi)$, where $\varphi$ is a closed 3-form. These structures arise naturally in string theory, where the closed 3-form $\varphi$ plays an important role. However, in certain physical contexts \cite{22}, the 3-form $\varphi$ may not be closed, motivating the study of more general structures in which $\varphi$ is allowed to be non-closed, governed by an additional differential form.

In this paper, we introduce and study $\theta$-almost twisted Poisson manifolds, defined as quadruples $(M,\Lambda,\varphi,\theta)$ consisting of a smooth manifold $M$, a bivector $\Lambda$, a 3-form $\varphi$, and a closed 1-form $\theta$, subject to the conditions $d\varphi = \theta \wedge \varphi$, $\Lambda^{\#}(\theta)=0$, and $\frac{1}{2}[\Lambda,\Lambda] = \Lambda^{\#}(\varphi)$. When $\theta=0$, we recover twisted Poisson manifolds; when, in addition, $\varphi=0$, we obtain ordinary Poisson manifolds. The closed 1-form $\theta$ controls the failure of $\varphi$ to be closed, and the condition $\Lambda^{\#}(\theta)=0$ ensures compatibility with the bivector field.

Our primary objective is to construct a Lie-Rinehart algebra on the module of 1-forms of a $\theta$-almost twisted Poisson manifold. This is achieved by introducing a novel bracket $[\ ,\ ]_{\varphi,\theta}$ that generalizes both the Koszul bracket and the bracket used in twisted Poisson geometry \cite{R11,R2}. The associated anchor map $\Lambda^{\#}$ then yields a cochain complex whose cohomology, termed $\theta$-almost twisted Poisson cohomology, generalizes the Lichnerowicz-Poisson and Lichnerowicz-twisted Poisson cohomologies. We explicitly compute this cohomology for a specific example on $\mathbb{R}^5$.

A secondary objective is the development of geometric quantization for $\theta$-almost twisted Poisson manifolds. We define a suitable contravariant derivative on a complex line bundle and establish a prequantization condition in terms of the cohomology class of the bivector $\Lambda$. This condition generalizes the well-known Kostant-Souriau condition for symplectic manifolds \cite{R22} and the condition for twisted Poisson manifolds due to Petalidou \cite{R2}. We then introduce a notion of polarization compatible with the $\theta$-almost twisted Poisson structure and construct the quantum Hilbert space by tensoring the prequantum line bundle with the bundle of half-densities. We illustrate the theory with examples, including a detailed quantization on $\mathbb{R}^5$ with specific choices of $\Lambda$, $\varphi$, and $\theta$.

The paper is organized as follows: Section \ref{sec1} recalls the necessary background on Poisson geometry, the Schouten-Nijenhuis bracket, and Lie-Rinehart algebras. Section \ref{sec2} introduces $\theta$-almost twisted Poisson manifolds, provides illustrative examples, constructs the Lie-Rinehart algebra, and defines the $\theta$-almost twisted Poisson cohomology. Section \ref{sec3} is devoted to geometric quantization: we define the contravariant derivative, introduce the $\theta$-almost twisted Poisson-Chern class, establish the prequantization condition, and construct the Hilbert space via polarization. A detailed example on $\mathbb{R}^5$ is provided.
	
	\section{Tools}\label{sec1}
Throughout this paper, we assume $k \geq 1$, all objects are smooth, and $M$ denotes a manifold. The algebra of smooth functions on $M$ is denoted by $C^\infty(M)$.
For $p \geq 0$, $\Omega^p(M)$ and $\mathfrak{X}^p(M)$ denote the spaces of $p$-forms and $p$-vector fields on $M$, respectively. In particular, $\Omega^0(M) = \mathfrak{X}^0(M) = C^\infty(M)$. We denote $\Omega(M) = \bigoplus_{p \geq 0} \Omega^p(M)$ and $\mathfrak{X}(M) = \bigoplus_{p \geq 0} \mathfrak{X}(M)$, where $\bigoplus$ denotes the direct sum.

	\subsection{Basic definitions}
Let $X\in \mathfrak{X}(M)$, and let $\psi\in\Omega^k(M)$. We denote by $i_X\psi$ the interior derivative of $\psi$ by $X$, which is defined by
\begin{align*}
	i_X\psi(X_1,\dots,X_{k-1})&=\psi(X,X_1,\dots,X_{k-1})
\end{align*}
for every $X_1,\dots,X_{k-1}\in \mathfrak{X}^1(M)$. Observe that the map $i_X:\Omega^k(M)\to\Omega^{k-1}(M)$ is a $C^\infty(M)$-linear endomorphism of $\Omega(M)$ of degree -1.

A bivector $\Lambda\in \mathfrak{X}^{2}(M)$ naturally induces a morphism $\Lambda^\#: \Omega^k(M)\longrightarrow \mathfrak{X}^k(M)$ defined by
\begin{align}\label{Eq3}
	\Lambda^\#(\psi)(\alpha_1,\dots,\alpha_k) &= (-1)^{k}\psi\bigg(\Lambda^\#(\alpha_1),\dots,\Lambda^\#(\alpha_k)\bigg),
\end{align}
for every $\psi\in\Omega^k(M)$ and $\alpha_1,\dots,\alpha_k\in \Omega^1(M)$.
In particular, $\Lambda^\#(f) = f$ for $f \in C^\infty(M)$, and for every $\alpha\in \Omega^1(M)$, $\Lambda^\#(\alpha)$ is a vector field defined by $\Lambda^\#(\alpha)(\beta)=\Lambda(\alpha,\beta)$ for $\beta\in \Omega^1(M)$. We denote $\Lambda^\#(\psi)$ by $\psi^{\#}$ for a differential form $\psi$ on $M$.
	\subsection{Koszul and Schouten brackets}
The Koszul bracket $[\ ,\ ]_K$ associated to a bivector field $\Lambda$ is defined by
$$[\alpha,\beta]_K=\mathcal{L}_{\Lambda^\#(\alpha)}(\beta)- \mathcal{L}_{\Lambda^\#(\beta)}(\alpha)-d(\Lambda(\alpha,\beta)), \quad\alpha,\beta\in \Omega^1(M).$$

The Schouten-Nijenhuis (or Schouten) bracket was introduced by Schouten (\cite{RS}) and Nijenhuis (\cite{RN}). Its definition stems from the following result.
\begin{theorem}\cite[Theorem 2.8]{R9}\label{ui}
	There exists a unique $\mathbb{R}$-bilinear operation $[\ ,\ ]: \mathfrak{X}^{p+1}(M)\times \mathfrak{X}^{q+1}(M)\longrightarrow \mathfrak{X}^{p+q+1}(M)$ for $p,q\geq -1$
	which satisfies the following properties:
	\begin{itemize}
		\item[(i)] When $p=q=0$, it is the usual Lie bracket of vector fields.
		\item[(ii)] When $p=0$ and $q=-1$, it is the Lie derivative $[X,f]=\mathcal{L}_Xf=X(f)$.
		
		\item For $P\in\mathfrak{X}^{p+1}(M)$, $Q\in \mathfrak{X}^{q+1}(M)$, and $R\in \mathfrak{X}^{r+1}(M)$,
		\begin{itemize}
			\item[(iii)] Graded skew-symmetry: $[P,Q]=-(-1)^{pq}[Q,P]$.
			\item[(iv)] Graded Leibniz identity: $[P,Q\wedge R]=[P,R]\wedge R+(-1)^{p(q+1)}Q\wedge[P,R]$.
			\item[(v)] Graded Jacobi identity: the graded Jacobi operator (Jacobiator) $J_{[,]}(P,Q,R)$ of $P,Q,R$ is identically null, where
			$$
			J_{[,]}(P,Q,R) = (-1)^{rp}[P,[Q,R]] + (-1)^{pq}[Q,[R,P]] + (-1)^{qr}[R,[P,Q]].
			$$
		\end{itemize}
	\end{itemize}
\end{theorem}

\subsection{Twisted Poisson manifold}	
\begin{definition}\label{expi2}\cite[pages 2-3]{R2}
	A twisted Poisson manifold is a manifold $M$ endowed with a bivector field $\Lambda$ and a closed 3-form $\varphi$ on $M$ such that
	\begin{equation*}
		\frac{1}{2}[\Lambda,\Lambda] =\Lambda^{\#} (\varphi).
	\end{equation*}
\end{definition}
The proof of the following result follows the same lines as that of Claim, pp. 192 in \cite{R20}.
\begin{proposition}\label{prop1}
	Let $(M,\Lambda,\varphi)$ be a twisted Poisson manifold. For every functions $f,g \in C^\infty(M)$, one has
	\begin{equation*}
		\left[(df)^{\#},(dg)^{\#}\right]-\big(d\{f,g\}\big)^{\#} = (i_{(dg)^{\#}}i_{(df)^{\#}}\varphi)^{\#}.
	\end{equation*}
\end{proposition}

	\subsection{Lie-Rinehart algebra}
Let $R$ be a commutative ring. Let $A$ be an algebra over $R$. We recall that a derivation of $A$ is a morphism $\delta: A\longrightarrow A$ of $R$-modules such that $\delta(ab)=\delta(a)b+a\delta(b)$ for all $a,b\in A$. We denote by $\operatorname{Der}(A)$ the set of derivations of $A$. It is well known that $\operatorname{Der}(A)$ is a Lie algebra over $R$ under the commutator of derivations $[\delta_1,\delta_2]=\delta_1\circ \delta_2-\delta_2\circ \delta_1$. The $R$-module $\operatorname{Der}(A)$ is well known to be an $A$-module if the algebra $A$ is commutative.

\begin{definition}\cite[page 5]{R7}
	Let $R$ be a commutative ring, and let $A$ be a commutative $R$-algebra (not necessarily with unit). Let $(L, [-,-])$ be a Lie algebra over $R$. A Lie-Rinehart algebra structure on $L$ is a Lie algebra homomorphism $\rho: L \longrightarrow \operatorname{Der}(A)$ satisfying the following compatibility properties:
	\begin{enumerate}
		\item $\rho(a\alpha)(b) = a\rho(\alpha)(b)$,
		\item $[\alpha, a\beta] = a[\alpha, \beta] + \rho(\alpha)(a)\beta$,
	\end{enumerate}
	where $a, b \in A$ and $\alpha, \beta \in L$. A Lie-Rinehart algebra is a pair $(L, \rho)$ where $\rho$ is a Lie-Rinehart structure on $L$.
\end{definition}
	
	\section{$\theta$-almost twisted Poisson cohomology}\label{sec2}
	In this section, we first introduce $\theta$-almost twisted Poisson manifolds before presenting their associated cohomology theory.
	\subsection{$\theta$-almost twisted Poisson structures}
	A $\theta$-almost twisted Poisson manifold is a manifold $M$ equipped with  a bivector field $\Lambda$,
	together with a 3-form $\varphi$ and a closed  1-form $\theta$ on $M$, such that the following hold: $d\varphi=\theta\wedge\varphi$,
	$\Lambda^\#(\theta)=0$,  and
	\begin{align}\label{Eq2}
	\dfrac{1}{2}[\Lambda,\Lambda] &=\Lambda^\#  (\varphi),
	\end{align}
	where $[\ ,\ ]$ is the Schouten-Nijenhuis bracket defined from Theorem~\ref{ui}.
	A $\theta$-almost twisted Poisson manifold   will be denoted by $(M,\Lambda,\varphi,\theta)$.
	
	This definition differs from that of a twisted Poisson manifold by relaxing the closure condition on the 3-form $\varphi$ through the introduction of a closed 1-form $\theta$, subject to specific compatibility constraints.
	
	The following examples illustrate $\theta$-almost twisted Poisson manifolds constructed from manifolds equipped with additional structures.\\[0.2cm]

\noindent\textbf{Examples.}
\begin{enumerate}
	\item Every Poisson manifold $(M, \Lambda)$ can be endowed with a twisted Poisson structure $(\Lambda, \varphi)$, and a twisted Poisson manifold $(M, \Lambda, \varphi)$ is a $0$-almost twisted Poisson manifold.
	
	\item Let $(M, \Lambda_0, \varphi_0)$ be a twisted Poisson manifold, and let $f \in C^\infty(M, \mathbb{R})$ be a Casimir function, i.e., $\Lambda_0^{\#}(df) = 0$. For $\Lambda = f\Lambda_0$, $\theta = -f^{-1} df$, and $\varphi = f^{-1} \varphi_0$, the quadruple $(M, \Lambda, \varphi, \theta)$ is a $\theta$-almost twisted Poisson manifold. A similar example can be constructed on $M \times \mathbb{R}$ with $\Lambda = e^t \Lambda_0$, $\varphi = e^{-t} \varphi_0$, and $\theta = -dt$, where $t$ is the canonical coordinate on $\mathbb{R}$. Moreover, $\varphi$ is non-closed on $M \times \mathbb{R}$.
	
	\item \label{exem1.1.5} Let $(x_1, x_2, x_3, x_4, x_5)$ be a coordinate system on $\mathbb{R}^5$. Let $\theta = dx_5$, $\Lambda = f \partial_{x_1} \wedge \partial_{x_2} + g \partial_{x_3} \wedge \partial_{x_4}$, and
	\begin{align*}
		\varphi &= (\partial_{x_1} g^{-1}) dx_1 \wedge dx_3 \wedge dx_4 + (\partial_{x_2} g^{-1}) dx_2 \wedge dx_3 \wedge dx_4 \\
		&\quad + (\partial_{x_5} g^{-1} - g^{-1}) dx_3 \wedge dx_4 \wedge dx_5 + (\partial_{x_3} f^{-1}) dx_1 \wedge dx_2 \wedge dx_3 \\
		&\quad + (\partial_{x_4} f^{-1}) dx_1 \wedge dx_2 \wedge dx_4 + (\partial_{x_5} f^{-1} - f^{-1}) dx_1 \wedge dx_2 \wedge dx_5,
	\end{align*}
	where $f$ and $g$ are nowhere vanishing functions on $\mathbb{R}^5$. Then $(\Lambda, \varphi, \theta)$ defines a $\theta$-almost twisted Poisson structure on $\mathbb{R}^5$, and $\varphi$ is non-closed.
\end{enumerate}

		\begin{fact}Let $m \leq 4$. If $(\Lambda, \varphi, \theta)$ is a $\theta$-almost twisted Poisson structure on an $m$-dimensional manifold $M$, then $\varphi$ is closed; that is, $(\Lambda, \varphi)$ is a twisted Poisson structure on $M$. Conversely, a twisted Poisson structure $(\Lambda, \varphi)$ on $M$ defines a $\theta$-parameter family $\left\{(\Lambda, \varphi, \theta) \mid \Lambda^{\#}(\theta)=0\right\}$ of $\theta$-almost twisted Poisson structures. In other words, for a fixed $m$-dimensional manifold $M$, there is a one-to-one correspondence between the set of twisted Poisson structures and the set of $\theta$-parameter families of $\theta$-almost twisted Poisson structures.
		\end{fact}

	\subsection{$\theta$-almost twisted Poisson complex}

For a twisted Poisson manifold $(M,\Lambda,\varphi)$, the associated bracket $[\ ,\ ]_{\varphi}$ defines a Lie algebra structure on $\Omega^1(M)$. Moreover, the pair $(\Omega^1(M),\Lambda^{\#})$ forms a Lie-Rinehart algebra, which induces a cochain complex whose associated cohomology is called Lichnerowicz-twisted Poisson cohomology (see \cite{R11,R2,R20}).

This subsection is devoted to defining $\theta$-almost twisted Poisson cohomology.

Let $[\ ,\ ]_{\varphi,\theta}: \Omega^1(M) \times \Omega^1(M) \longrightarrow \Omega^1(M)$ be the bracket associated with a $\theta$-almost twisted Poisson manifold $(M, \Lambda, \varphi, \theta)$, defined by
\begin{equation}
	[\alpha, \beta]_{\varphi, \theta} = [\alpha, \beta]_K + i_{\Lambda^{\#}(\beta)} i_{\Lambda^{\#}(\alpha)} \varphi + \Lambda(\alpha, \beta) \cdot \theta. \label{expresscrochet}
\end{equation}

This bracket induces a Lie-Rinehart algebra on $\Omega^1(M)$. Before proving this statement, let us establish some preliminary results.

We remark that the map
			\begin{eqnarray}
				\{\ ,\ \}: C^\infty(M)\times (C^\infty(M)&\longrightarrow& (C^\infty(M),\quad   \{f,g\}\mapsto \Lambda(df,dg)
			\end{eqnarray}is $\mathbb{R}$-bilinear and skew-symmetric. Moreover, for every $f,g,h \in C^\infty(M)$, we have
			\begin{align}\label{Jac}
				\{f,\{g,h\}\}+\{g,\{h,f\}\}+\{h,\{f,g\}\}&=\Lambda^{\#}(\varphi)(df,dg,dh).
			\end{align}

The pair $(C^\infty(M), \{ , \} )$ will be called a $\theta$-almost $\varphi$-twisted Poisson algebra.
		\begin{proposition}\label{inc4}
			Let  $\alpha, \beta\in\Omega^1(M)$, and $f,g \in C^\infty(M)$. We have
			\begin{align}\label{inc1}
				[f\alpha,g\beta]_{\varphi,\theta}&=fg[\alpha,\beta]_{\varphi,\theta}+\Lambda^{\#}(f\alpha)(g)\beta-\Lambda^{\#}(g\beta)(f)\alpha.
			\end{align}
		\end{proposition}
		\begin{proof}
			From \eqref{expresscrochet}, we have:  \begin{align*}
				[f\alpha,g\beta]_{\varphi,\theta}&=[f\alpha,g\beta]_K +
				i_{\Lambda^{\#}(g\beta)}i_{\Lambda^{\#}(f\alpha)}\varphi +
				\Lambda(f\alpha,g\beta)\cdot
				\theta\\
				&=\mathcal{L}_{\Lambda^{\#}(f\alpha)}(g\beta)-\mathcal{L}_{\Lambda^{\#}(g\beta)}(f\alpha)-d\Lambda(f\alpha,g\beta)\\
				&\quad+fgi_{\Lambda^{\#}(\beta)}i_{\Lambda^{\#}(\alpha)}\varphi +
				fg\Lambda(\alpha,\beta)\cdot
				\theta\\
				&=i_{\Lambda^{\#}(f\alpha)}d(g\beta)-i_{\Lambda^{\#}(g\beta)}d(f\alpha)+d\Lambda(f\alpha,g\beta)\\
				&\quad+fgi_{\Lambda^{\#}(\beta)}i_{\Lambda^{\#}(\alpha)}\varphi +
				fg\Lambda(\alpha,\beta)\cdot \theta\\
				 &=fg\big[i_{\Lambda^{\#}(\alpha)}d(\beta)-i_{\Lambda^{\#}(\beta)}d(\alpha)+d\Lambda(\alpha,\beta)\big]+\Lambda^{\#}(f\alpha)(g)\beta\\
				&\quad-\Lambda^{\#}(g\beta)(f)\alpha+
				fgi_{\Lambda^{\#}(\beta)}i_{\Lambda^{\#}(\alpha)}\varphi +
				fg\Lambda(\alpha,\beta)\cdot \theta\\
				&=fg[\alpha,\beta]_{\varphi,\theta}+\Lambda^{\#}(f\alpha)(g)\beta-\Lambda^{\#}(g\beta)(f)\alpha.
			\end{align*}
		This completes the proof.
		\end{proof}
		The bracket $[\ ,\ ]_{\varphi,\theta}$ is related to the anchor map $\Lambda^{\#}$. The precise formulation is as follows:
		\begin{proposition}\label{inc2}
		For every $f,g,u ,v,c\in C^\infty(M)$, we have
			\begin{align*}
				\Lambda^{\#}\big([udf,vdg]_{\varphi,\theta}\big)&=[\Lambda^{\#}(udf),\Lambda^{\#}(vdg)],
			\end{align*}
		\end{proposition}
		\begin{proof}
			Let $f,g,u ,v,c\in C^\infty(M)$. We get
			\begin{align}\label{inc3}
				\nonumber
				\Lambda^{\#}\big([udf,vdg]_{\varphi,\theta}\big)(c)&=\Lambda^{\#}\Big(uv[df,dg]_ {\varphi,\theta}+\Lambda^{\#}(udf)(v)dg-\Lambda^{\#}(vdg)(u)df\Big)(c)\\ \nonumber
				&=uv\Lambda^{\#}\Big([df,dg]_
				{\varphi,\theta}\Big)(c)+\Lambda^{\#}(udf)(v)\Lambda^{\#}(dg)(c)\\ \nonumber
				&\quad-\Lambda^{\#}(vdg)(u)\Lambda^{\#}(df)(c)\\ \nonumber
				&=uv\Lambda^{\#}\big(d\{df,dg\}+i_{\Lambda^{\#}(dg)}i_{\Lambda^{\#}(df)}\varphi\big)(c)\\ \nonumber
				&\quad+\Lambda^{\#}(udf)(v)\Lambda^{\#}(dg)(c)-\Lambda^{\#}(vdg)(u)\Lambda^{\#}(df)(c)\\ \nonumber
				&=uv[\Lambda^{\#}(df),\Lambda^{\#}(dg)](c)+\Lambda^{\#}(udf)(v)\Lambda^{\#}(dg)(c)\\
				&\quad-\Lambda^{\#}(vdg)(u)\Lambda^{\#}(df)(c)\\  \nonumber
				&=[\Lambda^{\#}(udf),\Lambda^{\#}(vdg)](c)\end{align}
			where equality \eqref{inc3} follows from Proposition \ref{prop1}.
			The proposition is then proved.
			\end{proof}
			
			In what follows, we define $\eta_{f,g,h}=i_{(df)^{\#}}i_{(dg)^{\#}}i_{(dh)^{\#}}\varphi$ and $\eta_{f,g}=i_{(df)^{\#}}i_{(dg)^{\#}}\varphi$ for every $f, g, h \in C^\infty(M)$. The following two results are crucial for defining a Lie-Rinehart structure on $\Omega^1(M)$, as they provide the necessary conditions for such a structure.
		\begin{lemma}\label{lemm1} For every real functions $f$, $g$, and $h$ on $M$,
			$$[df,\eta_{h,g}]_{\varphi,\theta}=i_{(df)^{\#}}d\eta_{h,g}+d\eta_{f,h,g}+ i_{[(dg)^{\#},(dh)^{\#}]}i_{(df)^{\#}}\varphi-\eta_{\{g,h\},f}+\eta_{f,h,g}\theta.$$
		\end{lemma}
		\begin{proof}	
			From \eqref{expresscrochet}, it follows that
			\begin{align*}
				[df,\eta_{h,g}]_{\varphi,\theta} &=[df,\eta_{h,g}]_K+i_{(\eta_{h,g})^{\#}}i_{(df)^{\#}}\varphi +\Lambda(df,\eta_{h,g})\theta\\
				&=\mathcal{L}_{(df)^{\#}}(\eta_{h,g})-\mathcal{L}_{(\eta_{h,g})^{\#}}(df)-d\Lambda(df,\eta_{h,g})
				+i_{(\eta_{h,g})^{\#}}i_{(df)^{\#}}\varphi +\eta_{f,h,g}\theta\\
				&=i_{(df)^{\#}}d\eta_{h,g}+di_{(df)^{\#}}\eta_{h,g}-di_{(\eta_{h,g})^{\#}}df-d\Lambda(df,\eta_{h,g})\\
				&\quad+i_{(\eta_{h,g})^{\#}}i_{(df)^{\#}}\varphi+\eta_{f,h,g}\theta\\
				&=i_{(df)^{\#}}d\eta_{h,g}+di_{(df)^{\#}}\eta_{h,g}+d\Lambda(df,\eta_{h,g})-d\Lambda(df,\eta_{h,g})\\
				&\quad+i_{(\eta_{h,g})^{\#}}i_{(df)^{\#}}\varphi+\eta_{f,h,g}\theta\\
				&=i_{(df)^{\#}}d\eta_{h,g}+di_{(df)^{\#}}\eta_{h,g}+i_{(\eta_{h,g})^{\#}}i_{(df)^{\#}}\varphi+\eta_{f,h,g}\theta.
			\end{align*}
			By Proposition~\ref{prop1}, we have $$(\eta_{h,g})^{\#}=[(dg)^{\#},(dh)^{\#}]-\big(d\{g,h\}\big)^{\#}.$$
			Therefore,
			\begin{equation}\label{sr3}
				\begin{array}{l}
					[df,\eta_{h,g}]_{\varphi,\theta}= i_{(df)^{\#}}d\eta_{h,g}+di_{(df)^{\#}}\eta_{h,g}+i_{[(dg)^{\#},(dh)^{\#}]}i_{(df)^{\#}}\varphi
					\\\hspace{2cm}-i_{(d\{h,g\})^{\#}}i_{(df)^{\#}}\varphi+\eta_{f,h,g}\theta
					i_{(df)^{\#}}d\eta_{h,g}+d\eta_{f,h,g}\\\hspace{2cm}+ i_{[(dg)^{\#},(dh)^{\#}]}i_{(df)^{\#}}\varphi-\eta_{\{g,h\},f}+\eta_{f,h,g}\theta.
				\end{array}
			\end{equation}	
			This completes the proof of Lemma~\ref{lemm1}.
			
		\end{proof}\par
		\begin{lemma}\label{lem2} For all  functions $f$, $g$ and $h$   on $M$, we have
			\begin{align*}
				[df,[dg,dh]_{\varphi,\theta}]_{\varphi,\theta} & =d\{f,\{g,h\}\}+\{f,\{g,h\}\}\theta+\eta_{\{g,h\},f}\\
				&\quad+[df,\eta_{h,g}]_{\varphi,\theta}+[df,\{g,h\}\theta]_{\varphi,\theta}.
			\end{align*}	
		\end{lemma}
		\begin{proof}
			From \eqref{expresscrochet}, we have
			\begin{align*}
				[dg,dh]_{\varphi,\theta} &= [dg,dh]_K+i_{(dh)^{\#}} i_{(dg)^{\#}}\varphi+ \Lambda(dg,dh) \theta\\
				&= d\{g,h\}+\eta_{h,g}+\{g,h\}\theta.
			\end{align*}
			Then,
			\begin{align*}
				%	\label{t3}
				\nonumber % Remove numbering (before each equation)
				[df,[dg,dh]_{\varphi,\theta}]_{\varphi,\theta}
				&= [df,d\{g,h\}]_{\varphi,\theta}+[df,\eta_{h,g}]_{\varphi,\theta} + [df,\{g,h\}\theta]_{\varphi,\theta} \\ \nonumber
				&=[df,d\{g,h\}]_K+i_{(d\{g,h\})^{\#}}i_{(df)^{\#}}\varphi +[df,\eta_{h,g}]_{\varphi,\theta}\\ \nonumber
				&\quad +\Lambda(df,d\{g,h\})\theta+[df,\{g,h\}\theta]_{\varphi,\theta}\nonumber\\
				&=d\{f,\{g,h\}\}+\eta_{\{g,h\},f}+\{f,\{g,h\}\}\theta \nonumber\\
				&\quad+[df,\eta_{h,g}]_{\varphi,\theta}+[df,\{g,h\}\theta]_{\varphi,\theta}.\nonumber
			\end{align*}
		\end{proof}
		Now we have the tools to prove the main result of this subsection, which is stated precisely as follows:
		\begin{theorem}\label{prop3}
			Let $(M,\Lambda,\varphi,\theta)$ be a $\theta$-almost twisted Poisson manifold. The bracket $[\ ,\ ]_{\varphi,\theta}$ defined in \eqref{expresscrochet} is  a  Lie bracket on $\Omega^1(M)$. Moreover, the morphism  $\Lambda^{\#} :\Omega^1(M)\longrightarrow \mathfrak{X}^1(M)$ defined  in $\eqref{Eq3}$ is  a Lie-Rinehart  structure on $\Omega^1(M)$.
		\end{theorem}
		\begin{proof}~
			\begin{itemize}
				\item Let us first prove that the bracket $[\ ,\ ]_{\varphi,\theta}$ defines a Lie algebra structure on $\Omega^1(M)$.
				The $\mathbb{R}$-bilinearity and skew-symmetry of the bracket $[\ ,\ ]_{\varphi,\theta}$ follow directly from its definition.
				We will now prove the Jacobi identity. Using equality \eqref{inc1} and Proposition \ref{prop1}, we obtain
				$$
				 [udf,[vdg,wdh]_{\varphi,\theta}]_{\varphi,\theta}+\circlearrowleft=uvw\big([df,[dg,dh]_{\varphi,\theta}]_{\varphi,\theta}+\circlearrowleft\big)
				$$
				for every $f,g,h,u,v,w\in C^\infty(M)$, where $\circlearrowleft$ indicates a cyclic summation over $f, g$, and $h$. Therefore, it suffices to prove the Jacobi identity for $[\ ,\ ]_{\varphi,\theta}$ on Pfaffian forms.
				
				From \eqref{Jac}, we get
				\begin{align}\label{S1}
					0 &=dJ_{\{\ ,\ \}}(f,g,h)-d\Lambda^{\#}(\varphi)(df,dg,dh),
				\end{align}
				where $J_{\{\ ,\ \}}(f,g,h)=\{f,\{g,h\}\}+\{g,\{h,f\}\}+\{h,\{f,g\}\}$.
				Applying Lemma \ref{lem2} and using Equation \eqref{S1}, we have
				\begin{align}\label{sr2}
					\nonumber
					0 &= [df,[dg,dh]_{\varphi,\theta}]_{\varphi,\theta}+[dg,[dh,df]_{\varphi,\theta}]_{\varphi,\theta}+[dh,[df,dg]_{\varphi,\theta}]_{\varphi,\theta}\\ \nonumber
					&\quad-[df,\eta_{h,g}]_{\varphi,\theta}-[dg,\eta_{f,h}]_{\varphi,\theta}-[dh,\eta_{g,f}]_{\varphi,\theta}\\ \nonumber
					&\quad-\eta_{\{g,h\},f}-\eta_{\{h,f\},g}-\eta_{\{f,g\},h}-d\Lambda^{\#}(\varphi)(df,dg,dh)\\ \nonumber
					&\quad-\{f,\{g,h\}\}\theta-\{g,\{h,f\}\}\theta-\{h,\{f,g\}\}\theta\\
					&\quad-[df,\{g,h\}\theta]_{\varphi,\theta}-[dg,\{h,f\}\theta]_{\varphi,\theta}-[dh,\{f,g\}\theta]_{\varphi,\theta}.
				\end{align}
				From \eqref{expresscrochet}, we have
				\begin{align}\label{sr1}
					\nonumber % Remove numbering (before each equation)
					[df,\{g,h\}\theta]_{\varphi,\theta} &=[df,\{g,h\}\theta]_K+i_{(\{g,h\}\theta)^{\#}}i_{(df)^{\#}}\varphi +\Lambda(df,\{g,h\}\theta)\theta\\ \nonumber
					&=\mathcal{L}_{(df)^{\#}}(\{g,h\}\theta)-\mathcal{L}_{(\{g,h\}\theta)^{\#}}(df)-d\Lambda(df,\{g,h\}\theta)\\ \nonumber
					&\quad+i_{\{g,h\}\theta)^{\#}}i_{(df)^{\#}}\varphi +\Lambda(df,\{g,h\}\theta)\theta\\ \nonumber
					&=i_{(df)^{\#}}d(\{g,h\}\theta)+ di_{(df)^{\#}}(\{g,h\}\theta)+\{g,h\}i_{\theta^{\#}}i_{(df)^{\#}}\varphi \\ \nonumber
					&\quad-\{g,h\}\theta^{\#}(df)\theta\\
					&=\{f,\{g,h\}\}\theta.
				\end{align}
				In the same way, we also have
				\begin{align}\label{sr10}
					[dg,\{h,f\}\theta]_{\varphi,\theta} &= \{g,\{h,f\}\}\theta,
				\end{align}
				and
				\begin{align} \label{sr11}
					[dh,\{f,g\}\theta]_{\varphi,\theta} &= \{h,\{f,g\}\}\theta.
				\end{align}
				
				According to Lemma~\ref{lemm1},  relations \eqref{sr1}, \eqref{sr10}, and  \eqref{sr11},  equality \eqref{sr2} becomes
				\begin{equation}\label{sr6}
					\begin{array}{l}
						 0=[df,[dg,dh]_{\varphi,\theta}]_{\varphi,\theta}+[dg,[dh,df]_{\varphi,\theta}]_{\varphi,\theta}+[dh,[df,dg]_{\varphi,\theta}]_{\varphi,\theta}\\
						\hspace{0.7cm}-i_{(df)^{\#}}d\eta_{h,g}-d\eta_{f,h,g}-i_{[(dg)^{\#},(dh)^{\#}]}i_{(df)^{\#}}\varphi\\
						\hspace{0.7cm}- i_{(dg)^{\#}}d\eta_{f,h}-d\eta_{g,f,h}-i_{[(dh)^{\#},(df)^{\#}]}i_{(dg)^{\#}}\varphi\\
						\hspace{0.7cm}-i_{(dh)^{\#}}d\eta_{g,f}-d\eta_{h,g,f}-i_{[(df)^{\#},(dg)^{\#}]}i_{(dh)^{\#}}\varphi\\
						\hspace{0.7cm}-d\pi^{\#}(\varphi)(df,dg,dh)-\eta_{f,h,g}\theta.
					\end{array}
				\end{equation}
				Thus, for every  vector field $X$ on $M$, we have	
				\begin{equation}
					\begin{array}{l}
						0=i_{X}\Big([df,[dg,dh]_{\varphi,\theta}]_{\varphi,\theta}+[dg,[dh,df]_{\varphi,\theta}]_{\varphi,\theta}\\
						\hspace{0.7cm}+[dh,[df,dg]_{\varphi,\theta}]_{\varphi,\theta} \Big)-i_{X}i_{(df)^{\#}}d\eta_{h,g} \\
						\hspace{0.7cm}-i_{X}d\eta_{f,h,g}-i_{X}\Big(i_{[(dg)^{\#},(dh)^{\#}]}i_{(df)^{\#}}\varphi\Big)\\
						\hspace{0.7cm}- i_{X}i_{(db)^{\#}}d\eta_{f,h}-i_{X}d\eta_{g,fh}-i_{X}\Big(i_{[(dh)^{\#},(df)^{\#}]}i_{(dg)^{\#}}\varphi\Big)\\
						 \hspace{0.7cm}-i_{X}i_{(dh)^{\#}}d\eta_{g,f}-i_{X}d\eta_{h,g,f}-i_{X}\Big(i_{[(df)^{\#},(dg)^{\#}]}i_{(dh)^{\#}}\varphi\Big)\\
						\hspace{0.7cm}-i_{X}\big(d\pi^{\#}(\varphi)(df,dg,dh)\big)-i_{X}\big(\eta_{f,h,g}\cdot\theta\big).
					\end{array}\label{sr7}	
				\end{equation}
				Furthermore,
				\begin{align}
					%		\begin{array}{l}
						i_{X}i_{(df)^{\#}}d\eta_{h,g}\nonumber &=(df)^{\#}\Big(\eta_{h,g}(X)\Big)-X\Big(\eta_{h,g}\big((df)^{\#}\big)\Big)\\\nonumber
						&\quad-
						\eta_{h,g}\Big(\big[(df)^{\#},X\big]\Big)\\
						&\begin{array}{l}=(df)^{\#}\Big(\varphi\big((dg)^{\#},(dh)^{\#},X\big)\Big)\\
							\hspace{0.4cm}-
							 X\Big(\varphi\big((df)^{\#},(dg)^{\#},(dh)^{\#}\big)\Big)\\\hspace{0.4cm}-\varphi\Big(\big[(df)^{\#},X\big],(dg)^{\#},(dh)^{\#}\Big).
						\end{array}\label{ss1}	
					\end{align}
					In the same way, we have
					\begin{align}\label{ss2}
						% Remove numbering (before each equation)
						\begin{array}{l}i_{X}i_{(dg)^{\#}}d\eta_{f,h}= (dg)^{\#}\Big(\varphi\big((dh)^{\#},(df)^{\#},X\big)\Big)\\
							\hspace{2.5cm}-
							X\Big(\varphi\big((dg)^{\#},(dh)^{\#},(df)^{\#}\big)\Big)\\
							\hspace{2.5cm}
							-\varphi\Big(\big[(dg)^{\#},X\big],(dh)^{\#},(df)^{\#}\Big),
						\end{array}
					\end{align}
					and
					\begin{align}\label{ss3}
						% Remove numbering (before each equation)
						\begin{array}{l}i_{X}i_{(dh)^{\#}}d\eta_{g,f}=(dh)^{\#}\Big(\varphi\big((df)^{\#},(dg)^{\#},X\big)\Big)\\
							\hspace{2.5cm}-
							X\Big(\varphi\big((dh)^{\#},(df)^{\#},(dg)^{\#}\big)\Big)\\
							\hspace{2.5cm}-\varphi\Big(\big[(dh)^{\#},X\big],(df)^{\#},(dg)^{\#}\Big).
						\end{array}
					\end{align}
					From \eqref{sr7},   \eqref{ss1}, \eqref{ss2},  and  \eqref{ss3}, we obtain
					\begin{align*}
						 0&=i_{X}\Big([df,[dg,dh]_{\varphi,\theta}]_{\varphi,\theta}+[dg,[dh,df]_{\varphi,\theta}]_{\varphi,\theta}+[dh,[df,dg]_{\varphi,\theta}]_{\varphi,\theta} \Big) \\
						&\quad-d\varphi\big((df)^{\#},(dg)^{\#},(dh)^{\#},X\big)+\theta\wedge\varphi\big((df)^{\#},(dg)^{\#},(dh)^{\#},X\big)\\
						 &=i_{X}\Big([df,[dg,dh]_{\varphi,\theta}]_{\varphi,\theta}+[dg,[dh,df]_{\varphi,\theta}]_{\varphi,\theta}+[dh,[df,dg]_{\varphi,\theta}]_{\varphi,\theta} \Big).
					\end{align*}
					This prove the Jacobi identity of $[\ ,\ ]_{\varphi,\theta}$.
					\item  Proposition \ref{inc4} shows  that  $[\alpha,f\beta]_{\varphi,\theta}=f[\alpha,\beta]_{\varphi,\theta}+\Lambda^{\#}(\alpha)(f).\beta $,
					for every  1-forms $\alpha$, $\beta$, and function $f$ on $M$.
					\item  Proposition \ref{inc2}  establishes  that the  $C^\infty(M)$-module $\Omega^1(M)$  endowed with a bracket $[\ ,\ ] _{\varphi,\theta}$
					is a  $\mathbb{R}$-Lie algebra homomorphic to  $\mathfrak{X}^1(M)$ endowed with its natural Lie bracket.
				\end{itemize}
				This ends the proof.
			\end{proof}

			The key ingredient in defining the $\theta$-almost twisted Poisson cohomology of $(M,\Lambda,\varphi,\theta)$ is the differential operator $\partial_{\varphi,\theta}$ associated with $\mathfrak{X}(M)$, defined as follows:
			$$
			\partial_{\varphi,\theta}^n:\mathfrak{X}^{n}(M) \longrightarrow  \mathfrak{X}^{n+1}(M), \quad v \longmapsto \partial_{\varphi,\theta}^n(v)
			$$
			with
			\begin{equation}\label{expi1}
				\begin{array}{r}
					\partial_{\varphi,\theta}^n(v)(\alpha_1,\dots,\alpha_{n+1}) = \sum_{i=1}^{n+1}(-1)^{i-1}(\alpha_i)^{\#}(v(\alpha_1,\dots,\widehat{\alpha_i},\dots,\alpha_{n+1}))\\  \\+ \sum_{1\leq i<j\leq n+1}(-1)^{i+j}v([\alpha_i,\alpha_j]_{\varphi,\theta},\alpha_1,\dots,\widehat{\alpha_i},\dots,\widehat{\alpha_j},\dots,\alpha_{n+1})
				\end{array}
			\end{equation}
			for every $\alpha_1,\dots,\alpha_{n+1}\in\Omega^1(M)$. The Jacobi identity for the bracket $[\ ,\ ]_{\varphi,\theta}$ ensures that $\partial_{\varphi,\theta} \circ \partial_{\varphi,\theta} = 0$, a crucial property for the cohomology theory to be well-defined. As is standard, terms with a hat symbol are omitted.

			\begin{definition}  The $\theta$-almost twisted Poisson cohomology of $(M,\Lambda,\varphi,\theta)$ is defined as the cohomology of the complex $(\mathfrak{X}(M), \partial_{\varphi,\theta})$, denoted by $H_{\theta-atP}^{*}(M,\Lambda,\varphi,\theta)$ or simply $H_{\theta-atP}^{*}(M)$ when there is no ambiguity. Specifically, for every positive integer $n$,
				\begin{equation}\label{t4}
					H_{\theta-atP}^{n}(M)=\dfrac{ker\Big(\partial_{\varphi,\theta}^{n}:\mathfrak{X}^{n}(M)\to \mathfrak{X}^{n+1}(M)\Big)}{Im\Big(\partial_{\varphi,\theta}^{n-1}:
						\mathfrak{X}^{n-1}(M)\to \mathfrak{X}^{n}(M)\Big)}
				\end{equation}
				where $\mathfrak{X}^{-1}(M)=\{0\}$ by convention. The cohomology class of any element $\mu \in \ker(\partial_{\varphi,\theta}^{n}:\mathfrak{X}^{n}(M)\to \mathfrak{X}^{n+1}(M))$ is denoted by $[\mu]_{\varphi,\theta}$.		
			\end{definition}
			
			Analogously to the approach in \cite[pages 41-43]{R19}, we have $\partial_{\varphi,\theta}\circ \Lambda^{\#}=-\Lambda^{\#}\circ d$. Then the following holds.
			\begin{proposition} Let $(M,\Lambda,\varphi,\theta)$ be a $\theta$-almost twisted Poisson manifold. The de Rham cohomology of $M$ is denoted by $H_{dR}^{*}(M,\mathbb{R})$. The cochain complex homomorphism $\Lambda^{\#}: (\Omega(M),d)\longrightarrow (\mathfrak{X}(M), \partial_{\varphi,\theta})$ induces a homomorphism in cohomology, also denoted by $\Lambda^{\#}$,
				\begin{equation}\label{ex1}
					\Lambda^{\#}: H_{dR}^{*}(M,\mathbb{R}) \longrightarrow  H_{\theta-atP}^{*}(M),\
					[\mu] \longmapsto [\Lambda^{\#}(\mu)]_{\varphi,\theta}.
				\end{equation}
				Moreover, if $\Lambda$ is nondegenerate, then $\Lambda^{\#}: H_{dR}^{*}(M,\mathbb{R}) \longrightarrow  H_{\theta-atP}^{*}(M)$ is an isomorphism.
			\end{proposition}

\begin{example}

			Let  $(x_1,x_2,x_3,x_4,x_5)$ be a   coordinate system on  $\mathbb{R}^5$. We define:
			\begin{align*}
				\Lambda &= f \partial_{x_1} \wedge \partial_{x_2} + g \partial_{x_3} \wedge \partial_{x_4},\quad \theta = dx_5,\\
				\varphi &= \partial_{x_1} g^{-1} dx_1 \wedge dx_3 \wedge dx_4 + \partial_{x_2} g^{-1} dx_2 \wedge dx_3 \wedge dx_4
				\\ &\quad+ (\partial_{x_5} g^{-1} - g^{-1}) dx_3 \wedge dx_4 \wedge dx_5
				\quad + \partial_{x_3} f^{-1} dx_1 \wedge dx_2 \wedge dx_3 \\ &\quad+ \partial_{x_4} f^{-1} dx_1 \wedge dx_2 \wedge dx_4
				+ (\partial_{x_5} f^{-1} - f^{-1}) dx_1 \wedge dx_2 \wedge dx_5,
			\end{align*}
			where $f$ and $g$ are nowhere vanishing functions on $\mathbb{R}^5$. The quadruplet $(\mathbb{R}^5,\Lambda,\varphi,\theta)$ is a  $\theta$-almost twisted Poisson manifold. After a rather tedious calculation, we obtain
			\begin{equation*}
				H_{\theta\text{-atp}}^p(\mathbb{R}^5) =\begin{cases}
					\begin{array}{lcl}
						C_{x^5}&\text{if} &p=0\\
						E_{\partial_{x_5}}&\text{if}& p=1\\
						0& \text{if}& p \geq 2
					\end{array}
				\end{cases}
			\end{equation*}
			with
			$$C_{x^5}=\{f\in C^\infty(\mathbb{R}^5), \;f(x_1,x_2,x_3,x_4,x_5)=f(x_5)\},\;  E_{\partial_{x_5}}= \{f \partial_{x_5},\; f \in C_{x^5}\}.$$
\end{example}

The $\theta$-almost twisted Poisson cohomology we have defined provides a natural framework for geometric quantization. Indeed, we will show that prequantization involves a cohomology class in $H_{\theta\text{-atp}}^2(M)$, generalizing the Kostant-Souriau condition. Furthermore, the Lie-Rinehart structure on 1-forms allows us to define an adapted contravariant derivative, which is an essential step in constructing the quantum Hilbert space.

\section{Geometric quantization}\label{sec3}

Geometric quantization, pioneered by Kostant \cite{R22} and Souriau, provides a powerful framework for constructing quantum systems from classical phase spaces modeled by symplectic manifolds. This procedure, encompassing prequantization and polarization, assigns Hilbert spaces to symplectic manifolds and Hermitian operators to classical observables. It was later extended to Poisson manifolds by Vaisman \cite{R19}, and subsequently to twisted Poisson manifolds by Petalidou \cite{R2}. In this section, we adapt this framework to $\theta$-almost twisted Poisson manifolds.

	\subsection{$\theta$-almost twisted Poisson-Chern class of a  complex line bundle over  $\theta$-almost twisted Poisson manifold}
	Let $(M,\Lambda,\varphi,\theta)$  be a $\theta$-almost twisted Poisson manifold,  $\xi: K\to M$  a complex line bundle over $M$, $\Gamma (K)$ the space of global cross sections of  $\xi: K\to M$, and $\operatorname{End}_\mathbb{C}(\Gamma (K))$ the space of complex  linear endomorphisms of $ \Gamma (K)$.
	\begin{definition}\label{Def1}
		A contravariant derivative  on  $\xi: K\to M$   is a $\mathbb{R}$-linear mapping  $D: \Omega^1(M)\to \operatorname{End}_\mathbb{C}(\Gamma (K))$   satisfying the  following  conditions:
		\begin{align}\label{Eq4}
		D_{(f\alpha)}=fD_\alpha, \text{ and }\ D_\alpha(fs)=fD_\alpha(s)+ \Lambda^\#(\alpha)(f)s
		\end{align}
		for every $\alpha\in \Omega^1(M)$,$f\in C^\infty(M)$,   and $s\in \Gamma (K)$.
	\end{definition}
	Let  $(M,\Lambda,\varphi,\theta)$  be a $\theta$-almost twisted Poisson manifold, and $h$  a Hermitian metric on a complex line bundle  $\xi: K\to M$. We say that a  contravariant derivative $D$ on  $\xi: K\to M$ is compatible with $h$ if
	\begin{align}\label{Eq6}
	\Lambda^\#(\alpha)(h(s_1,s_2))=h(D_\alpha s_1,s_2)+h(s_1,D_\alpha s_2)
	\end{align}
	for every $\alpha\in  \Omega^1(M)$,  and  $s_1,s_2\in \Gamma (K)$.
	\begin{definition}
		Let $\xi: K\to M$ be   a complex line bundle over a $\theta$-almost twisted Poisson manifold $(M,\Lambda,\varphi,\theta)$, and $D$ a
		contravariant derivative on $\xi: K\to M$.  The curvature of $D$ is the mapping  $C_D: \Omega^1(M)\times \Omega^1(M)\longrightarrow \operatorname{End}_\mathbb{C}(\Gamma (K)), (\alpha,\beta)\mapsto C_D(\alpha,\beta)$  where
		\begin{align}\label{Eq7}
		C_D(\alpha,\beta)=D_\alpha\circ D_\beta-D_\beta\circ D_\alpha-D_{[\alpha,\beta]_{\varphi,\theta}}.
		\end{align}
	\end{definition}
	From definition \ref{Def1}, and relation \eqref{Eq7}, we deduce that the curvature  $C_D$ of $D$ is bilinear over $C^\infty(M)$, and skew-symmetric. Thus, from the fact that  $\xi: K\to M$ is a complex line bundle  over M,  there exists a globally defined complex bivector  field $P_{C_D}= P_{{C_D}_1}+iP_{{C_D}_2}$ on $M$, with $P_{{C_D}_1},P_{{C_D}_2}\in \mathfrak{X}^2(M)$ such that
	\begin{align}\label{Eq8}
	C_D(\alpha,\beta)(s)= P_{C_D}(\alpha,\beta)s
	\end{align}
	for every $\alpha,\beta\in \Omega^1(M)$, and $s\in \Gamma (K)$.

	The cohomology operator $\partial_{\varphi,\theta}$ defined in \eqref{expi1} can be extended by linearity  to the module $\mathfrak{X}_{\mathbb{C}}(M)$ of complex multivectors on $M$.  This extension is explicitly defined by setting  for any
	$P=P_1+iP_2\in  \mathfrak{X}_{\mathbb{C}}(M)$
	$$\partial_{\varphi,\theta}P=\partial_{\varphi,\theta}P_1+i\partial_{\varphi,\theta}P_2$$  where  $P_1,P_2\in \mathfrak{X}(M)$.
	We show that  $\partial_{\varphi,\theta}^2=0$.  Therefore, we obtain the corresponding cohomology which will be called the \textit{complex $\theta$-almost twisted Poisson cohomology}, and will be   denoted by  $H_{\mathbb{C}^{\theta-atP}}^{*}(M)$.  Moreover, we have the following.
	
	\begin{theorem}\label{theo3.1}
		Let $\xi: K\to M$ be   a complex line bundle over a $\theta$-almost twisted Poisson manifold $M$,  $D$  a contravariant derivative on  $\xi: K\to M$,  $C_D$ the curvature of $D$, and $P_{C_D}$ the complex bivector field associated to $C_D$ \eqref{Eq8}. Then:
		\begin{enumerate}
			\item[(i)]  $P_{C_D}$ define a cohomology class  $[P_{C_D}]_{\varphi,\theta}$ in $H_{\mathbb{C}^{\theta-atP}}^{2}(M)$;
			\item[(ii)] $[P_{C_D}]_{\varphi,\theta}$ does not depend of the contravariant derivative $D$;
			\item[(ii)] If $D$ is compatible with a Hermitian metric $h$ on  $\xi: K\to M$, then  $P_{C_D}$ is purely imaginary.
		\end{enumerate}
	\end{theorem}
	
	\begin{proof}
		
		\noindent(i) Let \(s\) be a local nowhere vanishing section.  The bundle  $\xi: K\to M$  being complex 1-dimensional, for every 1-form  $\alpha$ on $M$,   $\frac{D_\alpha s}{s}$ is a complex number.  From \eqref{Eq4}, the function $\alpha\mapsto  \frac{D_\alpha s}{s}$ is $\mathbb{R}$-linear, and $C^\infty(M)$-linear. Then, there exists a unique complex local vector field $X = X_1 + iX_2$  with $X_1,X_2\in\mathfrak{X}^1(M)$  such that
		\begin{align}\label{Eq9}
		D_\alpha s&= X(\alpha) s, \quad \forall \alpha \in \Omega^1(M).
		\end{align}
		
		By using \eqref{Eq7},\  \eqref{Eq9}, and \eqref{expi1}, we get:
		\[
		\begin{aligned}
		\CD(\alpha,\beta)(s) &= (D_\alpha D_\beta - D_\beta D_\alpha - D_{[\alpha,\beta]_{\varphi,\theta}})(s) \\
		&= D_\alpha(X(\beta) s) - D_\beta(X(\alpha) s) - X[\alpha,\beta]_{\varphi,\theta}s \\
		&= X(\beta) D_\alpha s + \Lambda^\#(\alpha)(X(\beta)) s - X(\alpha) D_\beta s \\
		&\quad - \Lambda^\#(\beta)(X(\alpha)) s - X([\alpha,\beta]_{\varphi,\theta}) s \\
		&= \left[ \Lambda^\#(\alpha)(X(\beta)) - \Lambda^\#(\beta)(X(\alpha)) - X[\alpha,\beta]_{\varphi,\theta} \right] s\\
		&=\partial_{\varphi,\theta} X(\alpha,\beta)s.
		\end{aligned}
		\]
		From \eqref{Eq8}, we have
		\begin{align}\label{Eq10}
		\PCD &= \partial_{\varphi,\theta} X.
		\end{align}
		Therefore, $\partial_{\varphi,\theta} \PCD = \partial_{\varphi,\theta}^2(X)=0$.  Hence, $\PCD$ defines a cohomology class  \([\PCD]_{\varphi,\theta} \in H_{\mathbb{C}^{\theta-atP}}^{*}(M)\).
		
		\noindent(ii) Let \(\widetilde{D}\) be another contravariant derivative on $\xi: K\to M$  with associated  vector field \(\widetilde{X}\). That is, $P_{C_{\widetilde{D}}}= \partial_{\varphi,\theta} \widetilde{X}$. We  Define \(\widehat{X} = \widetilde{X} - X\). Then,
		$P_{C_{\widetilde{D}}}= \partial_{\varphi,\theta} \widetilde{X} = \partial_{\varphi,\theta} (X + \widehat{X}) = \PCD + \partial_{\varphi,\theta} \widehat{X}$.
		Thus, $P_{C_{\widetilde{D}}}$ and $\PCD$ differ by a coboundary, so $[ P_{C_{\widetilde{D} }}]_{\varphi,\theta} = [\PCD]_{\varphi,\theta}$.\\
		
		\noindent(iii) Let \(e\) be a local orthonormal basis of $\Gamma (K)$. The  Compatibility of $D$ with a Hermitian metric
		$h$ implies that
		\begin{align*}
		0 &= \Lambda^\#(\alpha)(h(e,e))\\
		&= h(D_\alpha e, e) + h(e, D_\alpha e)\\
		& = h(X(\alpha)e,e) + h(e,X(\alpha)e)\\
		&=X(\alpha)+ \overline{X(\alpha)}.
		\end{align*}
		
		Thus, \(X\) is purely imaginary.  Therefore, \(\PCD = \partial_{\varphi,\theta} X\) is purely imaginary.
	\end{proof}
	We get the following definition.
	\begin{definition}
		Let $\xi: K\to M$ be   a complex line bundle over a $\theta$-almost twisted Poisson manifold $(M,\Lambda,\varphi,\theta)$. Suppose that  $D$  is  a contravariant derivative with curvature $C_D$ such that $P_{C_D}$ is purely imaginary. Then the cohomology class   $[\frac{i}{2\pi}P_{C_D}]_{\varphi,\theta}$ in $H_{\theta-atP}^{2}(M)$ will be called  \textit{the first real $\theta$-almost twisted Poisson-Chern class} of  $\xi: K\to M$.
	\end{definition}
	
	We recall that for   a complex line bundle  $\xi: K\to M$ over a   differential manifold $M$, and  $\nabla$  a Hermitian connection  on  $\xi: K\to M$, there exists a purely imaginary closed 2-form $\Omega_\nabla$ such that
	\begin{align}
	C_\nabla(X,Y)(s)&=\Omega_\nabla(X,Y).s
	\end{align}
	for all $X,Y\in \mathfrak{X}^1(M)$, and $s\in \Gamma (K)$  where $C_\nabla$ is the curvature of the connection $\nabla$. \\ The real Chern class  $c_1(K,\mathbb{R})\in H_{dR}^2(M,R)$ is the \textit{integral} cohomology class $[\frac{i}{2\pi}\Omega_\nabla]$, see \cite{R22}.
	Note that the canonical
	injection $i : \mathbb{Z}\to\mathbb{R}$ induces a homomorphism $i: H_{dR}^2(M,\mathbb{Z})\longrightarrow H_{dR}^2(M,\mathbb{R})$. A class $[\alpha]\in H_{dR}^2(M,\mathbb{R})$  is called \textit{integral} if  it lies in the image  of $i$, see \cite{R22}.
	
	The following result show the relation between the first real Chern class and  the  first real $\theta$-almost twisted Poisson-Chern class of a complex line bundle over a $\theta$-almost twisted Poisson manifold.
	\begin{theorem}\label{theo3.2}
		Let $\xi: K\to M$ be   a complex line bundle over a $\theta$-almost twisted Poisson manifold $(M,\Lambda,\varphi,\theta)$, $\nabla$ a Hermitian connection  on $\xi: K\to M$, and $D$ the associated Hermitian contravariant derivative to $\nabla$ on $\xi: K\to M$, that is for any $\alpha\in \Omega^1(M)$, $D_\alpha=\nabla_{\Lambda^\#(\alpha)}$.   If $c_1(K,\mathbb{R})$, and $[\frac{i}{2\pi}P_{C_D}]_{\varphi,\theta}$ are respectively the first real Chern class and the first real $\theta$-almost twisted Poisson-Chern class of $\xi: K\to M$, then
		\begin{align}
		\Lambda^\#(c_1(K,\mathbb{R}))&=\left[\frac{i}{2\pi}P_{C_D}\right]_{\varphi,\theta}
		\end{align}
		where $ \Lambda^\#: H_{dR}^{*}(M,\mathbb{R}) \longrightarrow  H_{\theta-atP}^{*}(M)$ is the homomorphism  defined in \eqref{ex1}.
	\end{theorem}
	\begin{proof}
		Let \(\Omega_\nabla\) be  a   purely imaginary, and  closed 2-form associated to the connection $\nabla$. That is,
		\[
		C_\nabla(X,Y)(s) = \Omega_\nabla(X,Y)s, \quad \forall X,Y \in \mathfrak{X}^1(M),\;  s \in \Gamma(K).
		\]
		So  \(c_1(K, \mathbb{R}) = \left[\frac{i}{2\pi} \Omega_\nabla\right] \in H^2_{\dR}(M,\mathbb{R})\). For a local section \(s\) of $\xi: K\to M$,  we have
		$\nabla_Y s = \omega(Y) s$  where $\omega$ is the 1-form such that   $\Omega_\nabla = d\omega$, see ~\cite{R23}.
		Then,
		\begin{align*}
		D_\alpha s = \nabla_{\Lambda^\#(\alpha)} s= \omega(\Lambda^\#(\alpha)) s = -\Lambda^\#(\omega)(\alpha)s.
		\end{align*}
		Thus, the associated  vector field for \(D\) is \(X = -\Lambda^\#(\omega)\).
		
		From \eqref{Eq10},
		$\PCD = \partial_{\varphi,\theta} X = -\partial_{\varphi,\theta} (\Lambda^\#(\omega))$. And from \eqref{Eq11},
		$\PCD = -(-\Lambda^\#(d\omega)) = \Lambda^\#(\Omega_\nabla)$. Therefore:
		\begin{align*}
		\left[\frac{i}{2\pi} \PCD\right]_{\varphi,\theta} = \left[\frac{i}{2\pi} \Lambda^\#(\Omega_\nabla)\right]_{\varphi,\theta}= \Lambda^\#\left(\left[\frac{i}{2\pi} \Omega_\nabla\right]\right)= \Lambda^\#(c_1(K, \mathbb{R})).
		\end{align*}
	\end{proof}
	
	\subsection{Prequantization }
The prequantization step in geometric quantization aims to construct a faithful representation of the classical algebra of observables as operators acting on sections of a complex line bundle. Specifically, for a $\theta$-almost twisted Poisson manifold $(M,\Lambda,\varphi,\theta)$, we seek to associate with each smooth function $f\in C^\infty(M)$ a Hermitian operator $\hat{f}$ on the space of cross-sections $\Gamma(K)$ of a Hermitian line bundle $\xi: K\to M$ such that the map $f\mapsto \hat{f}$ satisfies $\widehat{\{f,g\}}=[\hat{f},\hat{g}] $, where the commutator $[\ ,\ ]$  is defined on  $\operatorname{End}_\mathbb{C}(\Gamma(K))$. This requires the construction of a line bundle $\xi: K\to M$ equipped with a suitable contravariant derivative $D$.

Let $(M,\Lambda,\varphi,\theta)$ be a $\theta$-almost twisted Poisson manifold, and let $D$ be a contravariant derivative on a complex line bundle $\xi: K\to M$. Let $\hat{f}$ be the representation of $f\in C^\infty(M)$ in $\operatorname{End}_\mathbb{C}(\Gamma(K))$ defined by
	\begin{align}\label{g1}
	\hat{f}(s)&=D_{df}s+2\pi ifs,\ \ s\in \Gamma(K).
	\end{align}
	Let $A$ be  the subset of $\operatorname{End}_\mathbb{C}(\Gamma(K))$ defined by $A=\{\ \hat{f}\in \operatorname{End}_\mathbb{C}(\Gamma(K))\ \slash \ f\in C^\infty(M)\ \}$.  We define on $A$ the following bracket
	\begin{align}
	\{\hat{f},\hat{g}\}_{\varphi,\theta}&=[\hat{f},\hat{g}]-D_{i_{\Lambda^{\#}(\beta)}i_{\Lambda^{\#}(\alpha)}\varphi}-\Lambda(df,dg)D_\theta
	\end{align}
	for every  $f,g\in C^\infty(M)$ where $[\hat{f},\hat{g}]=\hat{f}\circ \hat{g}-\hat{g}\circ\hat{f}$.
	We have de following Proposition.
	
	\begin{proposition}
		The hat representation  \ \  $\widehat{} \ \ : (C^\infty(M), \{\ ,\ \})\longrightarrow \operatorname{End}_\mathbb{C}(\Gamma(K)),\ f\mapsto \hat{f}$ is a homomorphism, that is
		\begin{align}\label{g4}
		\widehat{\{f,g\}}&=\{\hat{f},\hat{g}\}_{\varphi,\theta},\ \  f,g\in C^\infty(M).
		\end{align}
		This means
		\begin{align}\label{sl1}
		C_D(df,dg)&=-2\pi i\{f,g\}.
		\end{align}
	\end{proposition}
	
	\begin{proof}
		From equality \eqref{g1}, and Definition~\ref{Def1}, we get
		\begin{align}\label{g3}
		\nonumber
		[\hat{f},\hat{g}]&=\hat{f}\circ\hat{g}-\hat{g}\circ\hat{f}\\
		&= D_{df}\circ D_{dg}-D_{dg}\circ D_{df}+ 4\pi i\{f,g\}.
		\end{align}
		By using, \eqref{g1},~\eqref{g2},~\eqref{Eq7}, and \eqref{g3}, we have:
		\begin{align*}
		\widehat{\{f,g\}}s&=D_{d\{f,g\}}s+2\pi i\{f,g\}s \\
		&=D_{[df,dg]_{\varphi,\theta}}s-D_{i_{\Lambda^{\#}(dg)}i_{\Lambda^{\#}(df)}\varphi}s-\Lambda(df,dg)D_\theta s+2\pi i\{f,g\}s\\
		&=D_{df}\circ D_{dg}s-D_{dg}\circ D_{df}s -C_D(df,dg)s-D_{i_{\Lambda^{\#}(dg)}i_{\Lambda^{\#}(df)}\varphi}s\\
		&\quad- \Lambda(df,dg)D_\theta s+2\pi i\{f,g\}s\\
		&=-C_D(df,dg)s+[\hat{f},\hat{g}]-4\pi i\{f,g\}s-D_{i_{\Lambda^{\#}(dg)}i_{\Lambda^{\#}(df)}\varphi}s\\
		&\quad- \Lambda(df,dg)D_\theta s+2\pi i\{f,g\}s\\
		&=\{\hat{f},\hat{g}\}_{\varphi,\theta}s-C_D(df,dg)s-2\pi i\{f,g\}s.
		\end{align*}
		And the proposition follows.  \end{proof}
	
	\begin{definition}\label{g5}
		We say that a  $\theta$-almost twisted Poisson manifold  $(M,\Lambda,\varphi,\theta)$ is prequantizable  if there exists a Hermitian complex line bundle $\xi: K\to M$, the prequantization bundle, such that the representation operator defined in  \eqref{g1} holds  on $\Gamma(K)$ and  satisfy  \eqref{g4}.
	\end{definition}
	
	According to Definition~\ref{g5}, the prequantization problem of a $\theta$-almost twisted Poisson manifold  $(M,\Lambda,\varphi,\theta)$ has a solution if and only if, there exists a  Hermitian complex line bundle $\xi: K\to M$ endowed with a contravariant derivative $D$ whose curvature $C_D$ satisfies \begin{align}\label{g6}
	C_D=-2\pi i\Lambda.
	\end{align}
	Therefore,  the curvature $C_D$  is  purely imaginary. This implies that  $D$ is  compatible with the Hermitian complex line bundle $\xi: K\to M$.
	\begin{theorem}\label{g7}
		A $\theta$-almost twisted Poisson manifold  $(M,\Lambda,\varphi,\theta)$ is prequantizable if and only if there exists a vector field $Z$ on $M$, and a closed 2-form $\eta$ on $M$  representing an integral cohomology  class of $M$ satisfying
		\begin{align}\label{g8}
		\Lambda+\partial_{\varphi,\theta}(Z)=\Lambda^{\#}(\eta).
		\end{align}
	\end{theorem}
	\begin{proof}
		Suppose that the  $\theta$-almost twisted Poisson manifold  $(M,\Lambda,\varphi,\theta)$ is prequantizable.  Then, there exists a Hermitian complex line bundle $\xi: K\to M$ endowed with a contravariant derivative $D$ whose curvature $C_D$ satisfies $ C_D=-2\pi i\Lambda.$ Thus, $P_{C_D}=-2\pi i\Lambda$.

		Let \(\nabla\) be a Hermitian connection on $\xi: K\to M$ with curvature \(\Omega_\nabla\), and $\widetilde{D}_\alpha = \nabla_{{\Lambda^{\#}(\alpha)}}$. By Theorems~\ref{theo3.1}, we have
		$[P_{C_{\widetilde{D}}}]_{\varphi,\theta} = [\PCD]_{\varphi,\theta}$.
		So, there exists a vector field $W$  such that
		$P_{C_{\widetilde{D}}} = \PCD + \partial_{\varphi,\theta} W $.
		
		By Theorem~\ref{theo3.2} and Proposition~\ref{ex1}, we get
		$\Lambda^{\#}(\Omega_\nabla)=P_{C_{\widetilde{D}}}$.
		Therefore,
		\begin{align*}
		\Lambda^{\#}(\Omega_\nabla)=\PCD + \partial_{\varphi,\theta} W=-2\pi i\Lambda+ \partial_{\varphi,\theta} W.
		\end{align*}
		By setting $Z= \frac{i}{2\pi} W$ and  \(\eta = \frac{i}{2\pi} \Omega_\nabla\) which is real, closed, and integral, we obtain
		\begin{align*}
		\Lambda+\partial_{\varphi,\theta}(Z)=\Lambda^{\#}(\eta).
		\end{align*}

		Conversely,  let $Z$ be  a vector field and  \(\eta\) be  a closed  2-form  on $(M,\Lambda,\varphi,\theta)$ such that \eqref{g8}.  There exists Hermitian line bundle \(K\) with connection \(\nabla\) such that \(\Omega_\nabla = -2\pi i \eta\). We  define a contravariant derivative  for all $\alpha \in \Omega^1(M),s\in \Gamma(K)$ as follows:
		$$
		D_\alpha s = \nabla_{\Lambda^{\#}(\alpha)} s + 2\pi i \alpha(Z)s.
		$$
		Then,
		\[
		\begin{aligned}
		\CD(\alpha,\beta)(s) &=D_\alpha\circ D_\beta s-D_\beta\circ D_\alpha s-D_{{[\alpha,\beta]_{\varphi,\theta}}}s\\
		&=D_\alpha\Big(\nabla_{\Lambda^{\#}(\beta)} s + 2\pi i \beta(Z)s\Big)-D_\beta\Big( \nabla_{\Lambda^{\#}(\alpha)} s + 2\pi i \alpha(Z)s\Big)\\
		&\quad- \nabla_{\Lambda^{\#}([\alpha,\beta]_{\varphi,\beta})} s - 2\pi i[\alpha,\beta]_{\varphi,\beta}(Z)s\\
		&= \nabla_{\Lambda^{\#}(\alpha)}\Big( \nabla_{\Lambda^{\#}(\beta)} s + 2\pi i \beta(Z)s \Big)  + 2\pi i \alpha(Z)\Big(\nabla_{\Lambda^{\#}(\beta)} s + 2\pi i \beta(Z)s  \Big)\\
		&\quad- \nabla_{\Lambda^{\#}(\beta)}\Big( \nabla_{\Lambda^{\#}(\alpha)} s + 2\pi i \alpha(Z)s \Big)  - 2\pi i \beta(Z)\Big(\nabla_{\Lambda^{\#}(\alpha)} s + 2\pi i \alpha(Z)s  \Big)\\
		&\quad- \nabla_{[\Lambda^{\#}(\alpha),\Lambda^{\#}(\beta)]} s - 2\pi i[\alpha,\beta]_{\varphi,\beta}(Z)s\\
		&= \Omega_\nabla\big(\Lambda^{\#}(\alpha), \Lambda^{\#}(\beta)\big)s + 2\pi i  \Lambda^{\#}(\alpha)(\beta(Z))s\\
		&\quad - 2\pi i \Big[\Lambda^{\#}(\beta)(\alpha(Z))+[\alpha,\beta]_{\varphi,\beta}(Z)\Big]s\\
		&= -2\pi i \eta\big(\Lambda^{\#}(\alpha), \Lambda^{\#}(\beta)\big)s + 2\pi i \partial_{\varphi,\theta} Z(\alpha,\beta)s \\
		&= -2\pi i \Big[ \Lambda^{\#}(\eta)(\alpha,\beta) - \partial_{\varphi,\theta} Z(\alpha,\beta) \Big]s \\
		&= -2\pi i \Lambda(\alpha,\beta)s.
		\end{aligned}
		\]
Since the vector field $Z$ is real, the contravariant derivative $D$ is Hermitian. We conclude that $D$ satisfies the prequantization condition.
\end{proof}

The following remark highlights a variety of examples of prequantizable $\theta$-almost twisted Poisson manifolds.

\begin{remark}~ \label{nb3}
\begin{enumerate}
\item Let $(M,\Lambda,\varphi)$ be a twisted Poisson manifold. From  Theorem~\ref{g7}, it follows that $M$ is prequantizable as a $0$-almost twisted Poisson manifold if and only if there exists a vector field $Z$   and a  closed 2-form $\eta$  representing  an integral cohomology class of $M$  satisfying \eqref{g8}.  Conversely, a $0$-almost twisted Poisson manifold $(M,\Lambda,\varphi,O)$ is prequantizable if and only if the triplet $(M, \Lambda,\varphi)$ is prequantizable as a twisted Poisson manifold. Indeed,   $\partial_{\varphi,0}=\partial_{\varphi}$. 			 
\item Let $(M_0,\Lambda_0,\varphi_0)$ be a twisted Poisson manifold.  Consider the  $\theta$-almost twisted Poisson structure on $M=M_0\times\mathbb{R}$ defined by
$\Lambda=e^t\Lambda_0$, $\varphi=e^{-t}\varphi_0$,  and $\theta=-dt$  where $t$ is the  coordinate
			%   on  $\mathbb{R}$.
\begin{enumerate}
\item Assume  that $\Lambda_0$  is nondegenerate. Then the  induced  $\theta$-almost twisted Poisson manifold  $(M,\Lambda,\varphi, \theta)$  is not  prequantizable. Let's prove that the prequantization equation \eqref{g8} has no solutions on M.
Note that for any vector field  $Z$ on $M_0$, there exists a  differential 1-form $\alpha$ such that $Z=\Lambda_0^{\#}(\alpha)$.  Let $\omega$  be a differential  2-form such that $\Lambda_0=\Lambda_0^{\#}(\omega)$.  Therefore,
\begin{align}
\Lambda+\partial_{\varphi,\theta}(Z)=\Lambda^{\#}(\eta)
\end{align}
where $\eta=e^{-t}\omega-d(e^{-t}\alpha)$. The differential 2-form $\eta$ cannot be identically null. Indeed, suppose it identically null. One deduces that $\omega=-dt\wedge\alpha+d\alpha$. This contradicts the fact that $\omega$ belongs to  $\Omega^2(M_0)$. Observe that  $\eta$ is  not closed. It follows that   $(M,\Lambda,\varphi, \theta)$  is not  prequantizable.
\item
Assume  that the    twisted Poisson manifold $(M_0,\Lambda_0,\varphi_0)$  is exact in the sense of Petalidou. Then   the  induced  $\theta$-almost twisted Poisson manifold  $(M,\Lambda,\varphi, \theta)$ is prequantizable. In fact, if  the    twisted Poisson manifold $(M_0,\Lambda_0,\varphi_0)$ is exact, then there exists a vector field $Z_0$ on $M_0$ such that $\Lambda_0= \partial_{\varphi_0}(Z_0)$. By identifying $M_0$ with the submanifold   $\{0\}\times M_0$ of $M$, we can identify  $Z_0$  as a  vector  field on $M$. Thus,  we obtain $\partial_{\varphi,\theta}Z_0=e^t\partial_{\varphi}Z_0$. Hence, by setting $Z=-Z_0$, and $\eta=0$ which represents the integral cohomology class $[0]\in H_{dR}^2(M,\mathbb{R})$, we have the following equality:
\begin{align*}
\Lambda+\partial_{\varphi,\theta}(Z)&=e^t(\Lambda_0+\partial_{\varphi_0}(Z))=0=\Lambda^{\#}(\eta).
\end{align*}
This justifies that $(M,\Lambda,\varphi, \theta)$ is prequantizable.
\end{enumerate}
\item An exact   $\theta$-almost twisted Poisson manifold   is a $\theta$-almost twisted Poisson manifold  $(M,\Lambda,\varphi, \theta)$ such that there exists a vector field $X$ on $M$ satisfying $\Lambda= \partial_{\varphi,\theta}(X)$.  Thus, any exact   $\theta$-almost twisted Poisson manifold is prequantizable by setting  $Z=-X$, and $\eta=0$.
\item Consider the  $\theta$-almost twisted Poisson structure   $(\Lambda,\varphi,\theta )$  on  $\mathbb{R}^5$   defined by
$\Lambda=e^f\partial_{x_1}\wedge\partial_{x_2}+e^g\partial_{x_3}\wedge\partial_{x_4}$,\ $\theta=d{x_5}$, and \\ $\varphi=-\partial_{x_1}(g)e^{-g}d{x_1}\wedge d{x_3}\wedge d{x_4}-\partial_{x_2}(g)e^{-g}d{x_2}\wedge d{x_3}\wedge d{x_4}
-e^{-g}(\partial_{x_5} g-1)d{x_3}\wedge d{x_4}\wedge d{x_5}
-\partial_{x_3}(f)e^{-f}d{x_1}\wedge d{x_2}\wedge d{x_3}-\partial_{x_4}(f)e^{-f}d{x_1}\wedge d{x_2}\wedge d{x_4}-e^{-f}(\partial_{x_5} f-1)d{x_1}\wedge d{x_2}\wedge d{x_5}$ \
where  $f,g\in C^\infty(\mathbb{R}^5)$, and $(x_1,x_2,x_3,x_4,x_5)$ is  a coordinate system in $\mathbb{R}^5$.	
Assume that $f(x_1,x_2,x_3,x_4,x_5)=f(x_1,x_2)$ and $g(x_1,x_2,x_3,x_4,x_5)=g(x_3,x_4)$. Then,  $(\mathbb{R}^5,\Lambda,\varphi,\theta )$ is prequantizable. Indeed, if we take $Z=\partial_{x_5}$, and  $\eta= -e^{-f}\partial_{x_1}\wedge\partial_{x_2}-e^{-g}\partial_{x_3}\wedge\partial_{x_4}$, the  differential 2-form  $\eta$ is closed. Since the manifold $\mathbb{R}^5$ is contractible, then, $H_{\text{dR}}^2(\mathbb{R}^5)=0$. Therefore,  the closed 2-form $\eta$ is exact and represents the trivial cohomology class, which is integral.  A calculation shows that a prequantization  equation \eqref{g8} holds.
		\end{enumerate}
	\end{remark}

	\subsection{Construction of Hilbert space}

In this section, we address the second step of the geometric quantization process for $\theta$-almost twisted Poisson manifolds: the construction of a Hilbert space from the prequantization space $\Gamma(K)$, on which a suitable Lie subalgebra of observables will be represented. We introduce the notion of a polarization adapted to this setting and construct the quantum Hilbert space by tensoring the prequantization line bundle with the bundle of complex half-densities.

	Let $(M,\Lambda,\varphi,\theta)$ be a $\theta$-almost twisted Poisson manifold. We  consider the complexification $\Omega_{\mathbb{C}}^1(M)=\Omega^1(M)\otimes \mathbb{C}$ of $\Omega^1(M)$, and extend both the bracket $[\cdot,\cdot]_{\varphi,\theta}$, and anchor map $\Lambda^\#$ to complex-valued forms. Then, the  pair $(\Omega_{\mathbb{C}}^1(M),[\cdot,\cdot]_{\varphi,\theta})$ is a complex Lie algebra.  We will define a \emph{polarization} to be a complex Lie subalgebra $\mathcal{P}$ of   $(\Omega_{\mathbb{C}}^1(M),[\cdot,\cdot]_{\varphi,\theta})$ such that
	$$\Lambda(\alpha,\beta)=0\text{ for every }\alpha,\beta\in\Omega_{\mathbb{C}}^1(M).$$
	
	Given such  a polarization $\mathcal{P}$, we  set
	\[
	P(\mathcal{P}) = \left\{ f \in C^\infty(M) \,\middle|\, [df,\alpha]_{\varphi,\theta} \in \mathcal{P},\ \forall \alpha \in \mathcal{P} \right\} \text{ and } \left(P(\mathcal{P})\right)^2=P(\mathcal{P}) \times P(\mathcal{P}).
	\]
	Let $\Delta=\Delta(P(\mathcal{P}) \times P(\mathcal{P}))=\{(f,f),\; f\in P(\mathcal{P})\}$ and
	\[
	\widetilde{P(\mathcal{P})} = \left\{ (f,g) \in \left(P(\mathcal{P})\right)^2 \setminus \Delta,\; [i_{(dg)^\#}i_{(df)^\#}\varphi+\Lambda(df,dg)\theta, \alpha]_{\varphi,\theta} \in \mathcal{P},\, \forall \alpha \in \mathcal{P} \right\}.
	\]
	Note that $\widetilde{P(\mathcal{P})}$  is symmetric with respect to  $\Delta(P(\mathcal{P}) \times P(\mathcal{P}))$.  We denote by  $\mathcal{Q}(\mathcal{P})$, the projection  of $\widetilde{P(\mathcal{P})}$ onto  $P(\mathcal{P})$.

	\begin{proposition} The pair
		$(\mathcal{Q}(\mathcal{P}),\{\cdot,\cdot\})$ is a $\theta$-almost $\varphi$-twisted Lie  subalgebra of $(C^\infty(M),\{\cdot,\cdot\})$.
	\end{proposition}
	
	\begin{proof}
		Let $f,g \in \mathcal{Q}(\mathcal{P})$. It suffices to  show that $\{f,g\} \in \mathcal{Q}(\mathcal{P})$, meaning
		$\{f,g\} \in P(\mathcal{P})$, and there exist $h \in \mathcal{P}$ with  $h\neq\{f,g\}$ such that  $(\{f,g\},h) \in \widetilde{P(\mathcal{P})}$.
		
		Let us first prove  that  $\{f,g\} \in P(\mathcal{P})$.
		By the  definition of $[ , ]_{\varphi,\theta}$, we have
		\begin{align}\label{na4}
		d\{f,g\}&= [df,dg]_{\varphi,\theta} - \Big(i_{\Lambda^\#(dg)}i_{\Lambda^\#(df)}\varphi+\Lambda(df,dg)\theta\Big).
		\end{align}
		Then by linearity, we obtain  $$[d\{f,g\},\alpha]_{\varphi,\theta}= [[df,dg]_{\varphi,\theta},\alpha]_{\varphi,\theta}-[ i_{\Lambda^\#(dg)}i_{\Lambda^\#(df)}\varphi+\Lambda(df,dg)\theta, \alpha]_{\varphi,\theta}$$ for every $\alpha\in\mathcal{P}$. If $f=g$, then $\{f,g\}=0\in P(\mathcal{P})$.  Assume that $f\neq g$. Then  by construction, $(f,g)\in \widetilde{P(\mathcal{P})}$. Therefore, $[ i_{\Lambda^\#(dg)}i_{\Lambda^\#(df)}\varphi+\Lambda(df,dg)\theta, \alpha]_{\varphi,\theta}\in \mathcal{P}$. By using the Jacobi identity of $[ , ]_{\varphi,\theta}$, we get
		\[
		[[df,dg]_{\varphi,\theta}, \alpha]_{\varphi,\theta} = [df, [dg,\alpha]_{\varphi,\theta}]_{\varphi,\theta} + [dg, [\alpha,df]_{\varphi,\theta}]_{\varphi,\theta}.
		\]
		Since $f,g \in P(\mathcal{P})$, and  $\mathcal{P}$ is a Lie subalgebra, $[[df,dg]_{\varphi,\theta}, \alpha]_{\varphi,\theta}\in \mathcal{P}$. Therefore, $\{f,g\} \in P(\mathcal{P})$.
		\par
		Note that for every constant function $h\in C^\infty(M)$, we  have
		$[ i_{\Lambda^\#(dh)}i_{\Lambda^\#(d\{f,g\})}\varphi+\Lambda(d\{f,g\},dh)\theta, \alpha]_{\varphi,\theta}=[0,\alpha]_{\varphi,\theta}\in \mathcal{P}$.
		Thus, $(\{f,g\},h) \in \widetilde{P(\mathcal{P})}$. Therefore, $\{f,g\} \in \mathcal{Q}(\mathcal{P})$.
	\end{proof}
	The elements of $\mathcal{Q}(\mathcal{P})$  will be called the \emph{straightforwardly quantizable observables} of $(M,\Lambda,\varphi,\theta)$. This extends that given in \cite{R2}.
	
	Assume  $(M,\Lambda,\varphi,\theta)$ is  a  prequantizable  $\theta$-almost twisted Poisson manifold  with prequantization bundle $\xi:K \to M$, Hermitian metric $h$, and compatible contravariant derivative $D$ satisfying $C_D = -2\pi i \Lambda$.
	Let $\mathcal{D}$ be the  half-density bundle  associated to  a tangent bundle $TM$. Its cross sections  $\varrho$  are  called  \emph{half-density} of $M$, and   are  complex valued map on bases of $\Gamma(TM)$ satisfying:
	\[
	\varrho_x(e_xA_x) = \varrho_x(e_x)|\det A_x|^{1/2}
	\]
	for any $x \in M$, basis $e_x$ of $T_xM$, and $A_x \in \mathrm{GL}(T_xM)$.  Due to the fact that $\mathrm{GL}(T_xM)$ acts transitively on the set of basis of $T_xM$,  The bundle $\mathcal{D}$ is a complex line bundle. The Lie derivative $\mathcal{L}$ of $\varrho$ is defined as for tensor field on $M$. See \cite{R24,R25} for more details.
	
	By using \eqref{Eq4}, and  the properties of  Lie derivative $\mathcal{L}$, we can extend the contravariant derivative  $D$  to  $\operatorname{End}_{\mathbb{C}}(\Gamma(K \otimes \mathcal{D}))$ as follow
	\begin{align}\label{na1}
	D_\alpha(s \otimes \varrho)& = (D_\alpha s) \otimes \varrho + s \otimes \mathcal{L}_{\Lambda^\#(\alpha)}\varrho, \quad\quad \alpha \in \Omega_{\mathbb{C}}^1(M)\quad s \otimes \varrho \in \Gamma(K \otimes \mathcal{D}).
	\end{align}
	
	For any  $f\in C^\infty(M)$, we can also extend the  operator  $\hat{f}$  to $\operatorname{End}_{\mathbb{C}}(\Gamma(K \otimes \mathcal{D}))$ as
	\begin{align}\label{na2}
	\hat{f}(s \otimes \varrho)& = D_{df}(s \otimes \varrho) + 2\pi i f (s \otimes \varrho).
	\end{align}

	\begin{proposition}
		The extended operator satisfies the prequantization condition. That is,
		\[
		\widehat{\{f,g\}}(s \otimes \varrho) = \{\hat{f},\hat{g}\}_{\varphi,\theta}(s \otimes \varrho), \quad\quad  s \otimes \varrho \in \Gamma(K \otimes \mathcal{D}).
		\]
	\end{proposition}
	
	\begin{proof}
		Let $s \otimes \varrho \in \Gamma(K \otimes \mathcal{D})$. Remark that from \eqref{na1}, and \eqref{na2}  we get
		
		\begin{align}\label{na3}
		\hat{f}(s \otimes \varrho ) &= (\hat{f}(c))\otimes\varrho+s\otimes\mathcal{L}_{\Lambda^{\#}(df)}\varrho.
		\end{align}
		Then, using \eqref{na3},\eqref{na4}, the properties of Lie derivative $\mathcal{L}$, and that of the anchor map $\Lambda^{\#}$, the proposition follows.
	\end{proof}
	
	Let $\mathcal{H}_0 $ be a subspace of $\Gamma(K \otimes \mathcal{D})$  defined by
	\begin{align}\label{nb2}
	\mathcal{H}_0& = \left\{ s \otimes \varrho \in \Gamma(K \otimes \mathcal{D}) \,\middle|\, D_\alpha(s \otimes \varrho) = 0,\ \forall \alpha \in \mathcal{P} \right\}.
	\end{align}
	
	By a Bohr-Sommerfeld type condition, we can assume  that $\mathcal{H}_0 \neq \{0\}$, see \cite[Page 71-72]{R26}.
	
	\begin{proposition}\label{GQprop1}
		For any $f \in \mathcal{Q}(\mathcal{P})$, and $s \otimes \varrho \in \mathcal{H}_0$, we have $\hat{f}(s \otimes \varrho) \in \mathcal{H}_0$.
	\end{proposition}
	
	\begin{proof}
		Let  $\alpha \in \mathcal{P}$, and $s \otimes \varrho \in \mathcal{H}_0$. By using \eqref{na1}, \eqref{na2}, \eqref{Eq7}, and the prequantization condition \eqref{sl1}, we get
		\[
		D_\alpha(\hat{f}(s \otimes \varrho)) = \hat{f}(D_\alpha(s \otimes \varrho))-D_{[df,\alpha]_{\varphi,\theta}}(s \otimes \varrho).
		\]
		Then, we have
		\begin{align*}
		D_\alpha(\hat{f}(s \otimes \varrho))&=\hat{f}(0)-0=0.
		\end{align*}
		
		Because $[df,\alpha]_{\varphi,\theta} \in \mathcal{P}$,  since  $f \in \mathcal{Q}(\mathcal{P})$.
	\end{proof}
	Proposition~\ref{GQprop1} shows that the map: $\hat{f}\mid_{\mathcal{H}_0}: \mathcal{H}_0\longrightarrow \mathcal{H}_0$ is well-defined.
	Hence, $\mathcal{H}_0$ can be used as a \emph{quantization space} for $\mathcal{Q}(\mathcal{P})$.
	To construct a Hilbert space on $\mathcal{H}_0$, we define a scalar product on $\mathcal{H}_0$.  We distinguish two cases:
	
	\paragraph{Compact case} If $M$ is compact, we  define
	\begin{align}\label{nb1}
	\langle s_1 \otimes \varrho_1, s_2 \otimes \varrho_2 \rangle = \int_M h(s_1,s_2) \varrho_1 \bar{\varrho}_2
	\end{align}
	where $h$ is  a Hermitian metric on  $\xi: K\longrightarrow M$, and the bar denotes  the complex conjugation.
	This integral converges due to compactness, making $\mathcal{H}_0$   a pre-Hilbert space.
	\begin{proposition}
		The operators $\hat{f}$ for $f \in \mathcal{Q}(\mathcal{P})$ are anti-Hermitian with respect to this scalar product.
	\end{proposition}
	
	\begin{proof}
		Let $s_1 \otimes \varrho_1,s_2 \otimes \varrho_2\in \Gamma(K \otimes \mathcal{D})$.  According to \eqref{na3}, \eqref{nb1},  and  \eqref{Eq6}, we get
		\begin{align}\nonumber
		\langle \hat{f}(s_1 \otimes \varrho_1), s_2 \otimes \varrho_2 \rangle&+\langle s_1 \otimes \varrho_1,\hat{f}( s_2 \otimes \varrho_2) \rangle\nonumber\\ & \hspace{-0.375cm}= \langle\hat{f}(s_1)\otimes\varrho_1+s_1\otimes\mathcal{L}_{\Lambda^\#(df)}\varrho_1,s_2 \otimes \varrho_2\rangle\nonumber\\&+\langle s_1 \otimes \varrho_1,\hat{f}(s_2)\otimes\varrho_2+s_2\otimes\mathcal{L}_{\Lambda^\#(df)}\varrho_2\rangle
		\nonumber\\ &\hspace{-0.375cm}=
		\int_M\Big( h(\hat{f}(s_1),s_2) +h(s_1,\hat{f}(s_2) )\Big) \varrho_1 \bar{\varrho}_2\nonumber\\ & +
		 \int_Mh(s_1,s_2)\Big((\mathcal{L}_{\Lambda^\#(df)}\varrho_1)\bar{\varrho_2}+\varrho_1(\mathcal{L}_{\Lambda^\#(df)}\bar{\varrho_2})\Big)\nonumber\\ &\hspace{-0.375cm}=
		\int_M\Lambda^\#(df)\Big(h(s_1,s_2)\Big) \varrho_1 \bar{\varrho}_2
		\nonumber \\& +\int_Mh(s_1,s_2)\Big((\mathcal{L}_{\Lambda^\#(df)}\varrho_1)\bar{\varrho_2}+\varrho_1(\mathcal{L}_{\Lambda^\#(df)}\bar{\varrho_2})\Big)
		\nonumber \\ &\hspace{-0.375cm}=
		\int_M\mathcal{L}_{\Lambda^\#(df)}\Big(h(s_1,s_2)\varrho_1\bar{\varrho_2}\Big)
		=0\label{GQeq1}
		\end{align}
		where equality (\ref{GQeq1}) comes from the density version of Stokes' Theorem, see\cite{R27}.
	\end{proof}
	
	The Hilbert space $\mathcal{H}$ is the completion of $\mathcal{H}_0$. To obtain Hermitian operators representing quantum observables, we multiply the anti-Hermitian operators $\hat{f}$ by the imaginary unit $i$. This transformation scales the prequantization condition \eqref{g4} by the same constant factor.
	
	\paragraph{Non-Compact case} If $M$ is not compact, a more involved procedure is required to define a finite scalar product. We consider the real Lie subalgebra $\mathcal{P}_0$ of $(\Omega^1(M),[ , ]_{\varphi,\theta})$ whose complexification is  $\mathcal{P}\cap\bar{\mathcal{P}}$. We have that $\Lambda(\alpha,\beta)=0$ for all $\alpha,\beta\in\mathcal{P}_0$. We will assume that the distribution $\Lambda^{\#}(\mathcal{P}_0)$ defines a regular foliation $\mathcal{F}$ on $M$ with a Hausdorff leaf space $N=M/\mathcal{F}$.
	Under these conditions, the Hamiltonian vector field  $\Lambda^{\#}(df)$ of any observable $f\in\mathcal{Q}(\mathcal{P})$ is projectable  with respect to $\Lambda^{\#}(\mathcal{P}_0)$ onto $N$ (we have, $[\Lambda^{\#}(df),\Lambda^{\#}(\alpha)]=\Lambda^{\#}([df,\alpha]_{\varphi,\theta})\in\mathcal{P}_0,\forall \alpha\in\mathcal{P}_0)$. This allows us to restrict our attention to a specific type of half-densities on $M$. Specifically, we consider those that are pullbacks $\varpi^*\varrho_N$ of half-densities $\varrho_N$ on the leaf space $N$ via the canonical projection $\varpi: M\longrightarrow N$. For such sections, the action of the quantum operator simplifies to:
	\begin{align*}
	\hat{f}(s\otimes\varpi^*\varrho_N)&=\hat{f}(s)\otimes\varpi^*\varrho_N+s\otimes\mathcal{L}_{\Lambda^{\#}(df)}\varpi^*\varrho_N\\
	&= \hat{f}(s)\otimes\varpi^*\varrho_N+s\otimes\varpi^*(\mathcal{L}_{\varpi_*\Lambda^{\#}(df)}\varrho_N).
	\end{align*}
	We have that   $\mathcal{L}_{\Lambda^{\#}(\alpha)}\varpi^*\varrho_{iN}=0$, for all $\alpha\in \mathcal{P}_0$, i=1,2, since $\varpi^*\varrho_{1N},\varpi^*\varrho_{2N}$ are projectable on $N$. Using this fact, \eqref{Eq6}, \eqref{na1}, and \eqref{nb2}, we show that the density  $h(s_1,s_2)\varpi^*\varrho_{1N}\varpi^*\varrho_{2N}$ is projectable on $N$.
	Consequently, the scalar  product of two projectable sections $s_1\otimes\varpi^*\varrho_{1N},s_2\otimes\varpi^*\varrho_{2N}\in \mathcal{H}_0$ involves the projection of the quantity $h(s_1,s_2)\varpi^*\varrho_{1N}\varpi^*\varrho_{2N}$ to a 1-density $\delta_N$ on $N$. The scalar product is then given by integration over $N$:
	$$\langle s_1\otimes\varpi^*\varrho_{1N},s_2\otimes\varpi^*\varrho_{2N}\rangle=\int_N \delta_N.$$ By considering the subspace $\mathcal{H}^c_0\subset\mathcal{H}_0$ of projectable  sections with compact support in $N$, we obtain a pre-Hilbert space. The operators $\hat{f}$ remain anti-Hermitian  with respect to this scalar product. The Hilbert space of quantum states is finally constructed by completing  $\mathcal{H}^c_0$ and defining Hermitian operators as $i\hat{f}$, analogous to the compact case.\\[0.2cm]

	\begin{example}[Example on $\mathbb{R}^5$]
	%\begin{remark}
	Consider the prequantizable  $\theta$-almost twisted Poisson manifold $(\mathbb{R}^5,\Lambda,\varphi,\theta)$ of Remark~\ref{nb3} with additional conditions as follows:
	$\Lambda = e^f\partial_{x_1}\wedge\partial_{x_2}+e^g\partial_{x_3}\wedge\partial_{x_4}$,
	$\varphi = e^{-g}d{x_3}\wedge d{x_4}\wedge d{x_5}+e^{-f}d{x_1}\wedge d{x_2}\wedge d{x_5}$, and  $\theta = dx_5$
	where $f=f(x_1,x_2)$, and $g=g(x_3,x_4)$.  As we have seen, the solution of \eqref{g8} is $(Z,\eta)$ where
	$Z=\partial_{x_5}$, and $\eta=-e^{-f}\partial_{x_1}\wedge\partial_{x_2}-e^{-g}\partial_{x_3}\wedge\partial_{x_4}$.
	
	The prequantization bundle  of $(\mathbb{R}^5,\Lambda,\varphi,\theta)$ is the trivial complex line Bundle
	$K = \mathbb{R}^5 \times \mathbb{C}$ with the usual Hermitian metric $h$ and the contravariant derivative $D$ defined by
	\begin{align}\label{nb5}
	D_\alpha s &= \Lambda^\#(\alpha)s,
	\end{align}
	for all $\alpha\in \Omega^1(\mathbb{R}^5),s\in\Gamma(\mathbb{R}^5\times\mathbb{C})=C^\infty(\mathbb{R}^5,\mathbb{C})$.
	We can identify $\mathbb{R}^5$ to $\mathbb{C}^2\times\mathbb{R}$. Then, we set:  $z_k=x_{2k-1}+ix_{2k}$ and $\bar{z_k}=x_{2k-1}-ix_{2k}$   with $i=1,2$.
	This gives  $dx_{2k}=\dfrac{1}{2}(dz_k+d\bar{z_k})$,  $dx_{2k-1}=-\dfrac{i}{2}(dz_k-d\bar{z_k})$, $\partial_{x_{2k-1}}=\partial_{z_{k}}+
	\partial_{\bar{z_{k}}}$, and $\partial_{x_{2k}}=i(\partial_{z_{k}}-
	\partial_{\bar{z_{k}}})$. Thus, in the complex coordinates $(z_1,z_2,t)$ of $\mathbb{R}^5$, the pair $(\Lambda,\varphi)$ is written as follow:
	\begin{align*}
	\Lambda&=-2i[e^f\partial_{z_1}\wedge\partial_{\bar{z_1}}+e^g\partial_{z_2}\wedge\partial_{\bar{z_2}}],
	\end{align*}
	\begin{align*}
	\varphi&=-\dfrac{i}{2}[e^{-f}dz_1\wedge d\bar{z_1}+e^{-g}dz_2\wedge d\bar{z_2}]\wedge dt.
	\end{align*}
	A convenient complex polarization  of $(\mathbb{R}^5,\Lambda,\varphi,\theta)$ is  $\mathcal{P} = \operatorname{span}\{ dz_1,dz_2 \}$,
	and we see that  $P(\mathcal{P})$ consists of the functions $h\in C^\infty(\mathbb{R}^5)$ such that
	$$[dh,dz_k]_{\varphi,\theta}\in \mathcal{P},\quad\quad i=1,2.$$  After a computation, $[dh,dz_k]_{\varphi,\theta}\in \mathcal{P}$ if and only if,
	\begin{align}\label{nb4}
	\partial_{\bar{z_k}}h&=0,\quad k=1,2.
	\end{align}
	We consider $\widetilde{P(\mathcal{P})}$ the set of pair $(h_1,h_2)$ of solutions of \eqref{nb4} for that
	$$[i_{\Lambda^{\#}(dh_2)}i_{\Lambda^{\#}(dh_1)}\varphi+\Lambda(dh_1,dh_2)dt, dz_k]_{\varphi,\theta}\in\mathcal{P},\quad k=1,2.$$
	We take $\mathcal{Q}(P)$ to be the projection of $\widetilde{P(\mathcal{P})}$ on $P(\mathcal{P})$. Remark that $t\in\mathcal{Q}(P)$.\par  To determine the  quantization space $\mathcal{H}_0$, we need the bundle of complex half-densities $\mathcal{D}$ over  $\mathbb{R}^5=\mathbb{C}^2\times\mathbb{R}$.
	A basis of $\mathcal{D}$ can be written as   $\beta=|v|^{\frac{1}{2}}$  where
	$$v=dx_1\wedge\dots\wedge dx_5=-\frac{1}{4}dz_1\wedge dz_2\wedge d\bar{z_1}\wedge d\bar{z_2}\wedge dt.$$
	If we take $1$ as a basis of $K=\mathbb{R}^5 \times \mathbb{C}$, then any section $s\otimes\varrho$ of  $K\otimes\mathcal{D}$ can be written as
	$s\otimes\varrho=1\otimes(\chi\beta)$,  where $\chi\in C^\infty(\mathbb{R}^5,\mathbb{C})$.\par
	If we denote by $D$ the extension \eqref{na1} of the  Hermitian contravariant derivative $D$ given by \eqref{nb5}, we obtain
	\begin{align}
	D_{dz_k}(1\otimes(\chi\beta))&=1\otimes \mathcal{L}_{\Lambda^{\#}(dz_k)}(\chi\beta),\quad k=1,2.
	\end{align}
	Therefore, Using the fact that $\mathcal{L}_X\gamma=(\frac{1}{2}div X)\gamma$ (see, for instance,\cite{R25}), it follows that
	$$D_{dz_k}(1\otimes(\chi\beta))=0$$ if and only if
	\begin{align}\label{nb6}
	\mathcal{L}_{\Lambda^{\#}(dz_k)}(\chi)\beta+\frac{\chi}{2}div \Lambda^{\#}(dz_k)\beta&=0.
	\end{align}
	
	But $\Lambda^{\#}(dz_k)=-2ie^{A_k}\partial_{\bar{z_k}}$   and  $div\Lambda^{\#}(dz_k)=-2ie^{A_k}\partial_{\bar{z_k}}A_k$, with $A_1=f, A_2=g$.
	Hence, equation \eqref{nb5} is equivalent to $$ \partial_{\bar{z_k}}\chi+ \frac{\chi}{2}\partial_{\bar{z_k}}A_k=0.$$
	Thus the quantization space $\mathcal{H}_0$ can be identified  with the space
	$$\mathcal{H}_0=\left\{\chi\in C^\infty(\mathbb{R}^5,\mathbb{C})\mid \quad \partial_{\bar{z_k}}\chi+ \frac{\chi}{2}\partial_{\bar{z_k}}A_k=0,\quad
	\forall k=1,2\quad \right\}.$$
	Note that  $e^{-\frac{1}{2}f+t}$, and  $e^{-\frac{1}{2}g+t}$ are  elements of  $\mathcal{H}_0$. This means
	$\mathcal{H}_0$ is non-trivial.
	
	Let $h\in \mathcal{Q}(P)$, and $\chi\in \mathcal{H}_0$. Using  \eqref{g1},\eqref{na2}, we deduce that
	$$ \hat{h}(\chi)=2\pi i h\chi+\Lambda(dh,d\chi)+\dfrac{\chi}{2}div\Lambda^{\#}(dh).$$ Moreover, the scalar product of two function
	$\chi_1,\chi_2\in \mathcal{H}_0$ with compact support is $$\langle \chi_1,\chi_2\rangle=\int_{\mathbb{R}^5}\chi_1\overline{\chi_2}\beta.$$
\end{example}

\end{document}